\documentclass{article}
\usepackage{amsfonts,amsmath,amssymb,amsthm}

\usepackage{hyperref}

\usepackage{graphics, epsfig}

\usepackage{titling}
\predate{}
\postdate{}

\title{Symmetric Kapranov and symmetric tropical ranks}
\author{Dylan Zwick \\ \href{mailto:zwick@math.utah.edu}{zwick@math.utah.edu}}
\date{}

\begin{document}

\maketitle

\begin{abstract} 
  This paper proves the $r \times r$ minors of an $n \times n$ symmetric matrix of indeterminates are a tropical basis when $r = 2$, $r = 3$, or $r = n$, and are not when $4 < r < n$ or $r = 4, n > 12$. In the process, it introduces two new notions of rank for symmetric matrices coming from tropical geometry, the symmetric tropical and the symmetric Kapranov rank, which are the symmetric versions of their standard counterparts defined by Develin, Santos, and Sturmfels.
\end{abstract}

In this paper, we investigate the question of when the minors of a symmetric matrix of indeterminates form a tropical basis. In Section 1 we review the basic concepts from tropical geometry required to understand the rest of the paper. In Section 2 we define symmetric analogs of the tropical rank and Kapranov rank defined by Develin, Santos, and Sturmfels \cite{dss}, and investigate some basic properties of these analogs. In Section 3 we prove a number of cases where the minors do form a tropical basis. In Section 4 we prove a number of cases where the minors do not form a tropical basis. In Section 5 we conclude with some open questions related to this paper. The main results of this paper can be summarized in the following theorem:

\newtheorem{thm}{Theorem}
\begin{thm}
  The $r \times r$ minors of an $n \times n$ symmetric matrix of indeterminates form a tropical basis when $r = 2$, $r = 3$, or $r = n$, and do not form a tropical basis when $4 < r < n$, or when $r = 4$ and $n > 12$.
\end{thm}

The author would like to thank the mathematics department of the University of Utah for support during the research for this paper, and in particular his advisor Aaron Bertram. The author would also like to thank Melody Chan for asking the question that inspired this paper, and for pointing out to the author that the cocircuit matrix of the Fano matroid could be rearranged into a symmetric matrix.

\section{Tropical preliminaries}

This section introduces the basic ideas from tropical geometry used in this paper, along with the results for general matrices that informed and inspired this investigation of symmetric matrices.

\subsection{Tropical basics}

The \emph{tropical semiring} $(\mathbb{R}, \oplus, \odot)$, is defined as the semiring with arithmetic operations:
\begin{center}
  $a \oplus b := min(a,b)$ \hspace{.1 in} and \hspace{.1 in} $a \odot b := a + b$.
\end{center}

A \emph{tropical monomial} $X_{1}^{a_{1}} \cdots X_{m}^{a_{m}}$ is a symbol, and represents a function equivalent to the linear form $\sum_{i} a_{i}X_{i}$ (standard addition and multiplication). 

A \emph{tropical polynomial} is a tropical sum of tropical monomials  
  \begin{center}
    $F(X_{1},\ldots,X_{m}) := \bigoplus_{a \in \mathcal{A}}C_{a}X_{1}^{a_{1}}X_{2}^{a_{2}} \cdots X_{m}^{a_{m}}$, \hspace{.1 in} with $\mathcal{A} \subset \mathbb{N}^{m}$, $C_{a} \in \mathbb{R}$
  \end{center}
  (tropical addition and multiplication), and represents a piecewise linear convex function $F: \mathbb{R}^{m} \rightarrow \mathbb{R}$. 

In this paper, tropical polynomials will be represented with upper case letters, while standard polynomials will be lower case.

The \emph{tropical hypersurface} $\textbf{V}(F)$ defined by a tropical polynomial $F$ is the locus of points $P \in \mathbb{R}^{m}$ such that at least two monomials in $F$ are minimal at $P$. This is also called the \emph{double-min locus} of $F$.

For example, the tropical hypersurface defined by the tropical polynomial
\begin{center}
  $X \oplus Y \oplus 0  = min\{x,y,0\}$
\end{center}
would include the point $(1,0)$, as both $Y$ and $0$ are minimal at that point, but would not include the point $(-1,0)$, as $X$ is uniquely minimal at that point. This is an example of a tropical line.

\vspace{.1 in}
\begin{tabular}{c}
  \centering
  \hspace{1in}\includegraphics[scale=1]{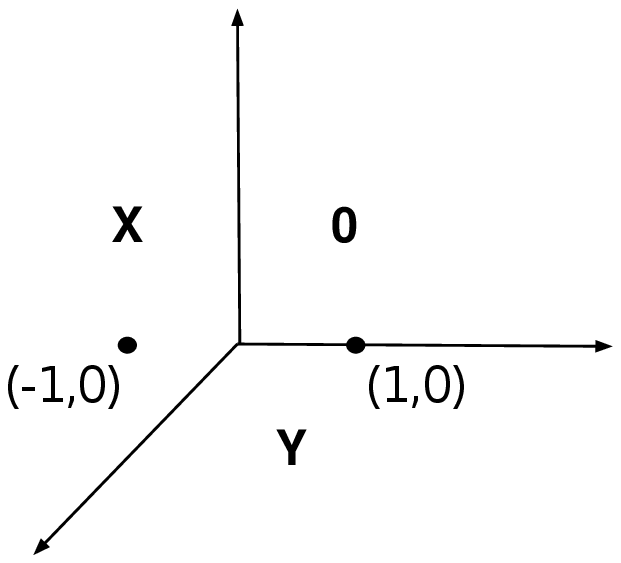}
\end{tabular}

\subsection{Tropical bases}

Let $k$ be an algebraically closed field. Let $f \in k[x_{1},\ldots,x_{m}]$ be a polynomial. The locus of points $p \in k^{m}$ such that $f(p) = 0$ is a \emph{hypersurface}, and is denoted $\textbf{V}(f)$. Let $I$ be an ideal of $k[x_{1},\ldots,x_{m}]$. The ideal $I$ defines a \emph{algebraic variety}, (or \emph{variety}, for short) $\textbf{V}(I)$, in $k^{m}$, which is the set of points $p \in k^{m}$ such that $f(p) = 0$ for all $f \in I$. If $I = (f_{1},\ldots,f_{n})$  then the set $\{f_{1},\ldots,f_{n}\}$ is a \emph{basis} for $I$, and $\textbf{V}(I)$ is equal to the locus of points $p \in k^{m}$ such that $f_{i}(p) = 0$ for all $f_{i}$ in the basis. Put succinctly
\begin{center}
  $\textbf{V}(I) = \bigcap \textbf{V}(f_{i})$.
\end{center}
So, a variety is an intersection of hypersurfaces. By the Hilbert basis theorem every ideal of $k[x_{1},\ldots,x_{m}]$ is finitely generated, so any variety is a finite intersection of hypersurfaces.

In the tropical setting there is an analog of a hypersurface, and we would like an analog of a variety. It might seem natural to define a tropical variety as the intersection of a finite set of tropical hypersurfaces, but these sets do not always have the properties we need in order for them to be useful analogs of algebraic varieties, and we instead call these sets tropical prevarieties. 

A \emph{tropical prevariety} $\textbf{V}(F_{1},\ldots,F_{n})$ is a finite intersection of tropical hypersurfaces:  
\begin{center}
  $\textbf{V}(F_{1},\ldots,F_{n}) = \bigcap_{i = 1}^{n} \textbf{V}(F_{i})$.
\end{center}

A tropical variety is defined differently. Let $K = \mathbb{C}\{\{t\}\}$ be the set of formal power series $a = c_{1}t^{a_{1}} + c_{2}t^{a_{2}} + \cdots$, where $a_{1} < a_{2} < a_{3} < \cdots$ are rational numbers that have a common denominator. These are called Puiseux series, and this set is an algebraically closed field of characteristic zero. For any nonzero element $a \in K$ define the degree of $a$ to be the value of the leading exponent $a_{1}$. This gives us a degree map $deg : K^{*} \rightarrow \mathbb{Q}$. For any two elements $a,b \in K^{*}$ we have
\begin{center}
  $deg(ab) = deg(a) + deg(b) = deg(a) \odot deg(b)$.
\end{center}
Generally, we also have
\begin{center}
  $deg(a + b) = min(deg(a),deg(b)) = deg(a) \oplus deg(b)$.
\end{center}
The only case when this addition relation is not true is when $a$ and $b$ have the same degree, and the coefficients of the leading terms cancel.

We would like to do tropical arithmetic over $\mathbb{R}$, and not just over $\mathbb{Q}$, so we enlarge the field of Puisieux series to allow this. Define the set $\tilde{K}$ by
\begin{center}
  $\tilde{K} = \left\{\sum_{\alpha \in A} c_{\alpha}t^{\alpha} | A \subset \mathbb{R} \text{ well-ordered}, c_{\alpha} \in \mathbb{C}\right\}$.
\end{center}
This is the set of Hahn series, and it is an algebraically closed field of characteristic zero containing the Puisieux series. We define a tropical variety in terms of a variety over $\tilde{K}$.

The degree map on $(\tilde{K}^{*})^{m}$ is the map $\mathcal{T}$ taking points $(p_{1},\ldots,p_{m}) \in (\tilde{K}^{*})^{m}$ to points $(deg(p_{1}),deg(p_{2}),\ldots,deg(p_{m})) \in \mathbb{R}^{m}$. A tropical variety is the image of a variety in $(\tilde{K}^{*})^{m}$ under the degree map. We call this image the \emph{tropicalization} of a set of points in $(\tilde{K}^{*})^{m}$. The tropicalization of a polynomial $f \in \tilde{K}[x_{1},\ldots,x_{m}]$ is the tropical polynomial $\mathcal{T}(f)$ formed by tropicalizing the coefficients of $f$, and converting addition and multiplication into their tropical counterparts. For example, the tropicalization of the polynomial 
\begin{center}
  $f = 3t^{2}xy - 7tx^{3}$ 
\end{center}
is the tropical polynomial
\begin{center}
  $\mathcal{T}(f) = 2XY \oplus 1X^{3}$.
\end{center}

In an unpublished manuscript, Mikhail Kapranov proved the following useful and fundamental result.

\begin{thm}[Kapranov's Theorem]
  For $f \in \tilde{K}[x_{1},\ldots,x_{m}]$ the tropical variety $\mathcal{T}(\textbf{V}(f))$ is equal to the tropical hypersurface $\textbf{V}(\mathcal{T}(f))$ determined by the tropical polynomial $\mathcal{T}(f)$.
\end{thm}

Given Kapranov's theorem if $I = (f_{1},\ldots,f_{n})$, then obviously the tropical prevariety determined by the set of tropical polynomials $\{\mathcal{T}(f_{1}),\ldots,\mathcal{T}(f_{n})\}$ contains the tropical variety determined by $I$:
\begin{center}
  $\mathcal{T}(\textbf{V}(I)) \subseteq \bigcap_{i = 1}^{n} \textbf{V}(\mathcal{T}(f_{i}))$.
\end{center}

While Kapranov's theorem gives us the two sets are equal if $n = 1$, in general the containment may be strict. For example, the lines in $(\tilde{K}^{*})^{2}$ defined by the linear equations
\begin{center}
  $f = 2x + y + 1$, \hspace{.1 in} and \hspace{.1 in} $g = tx + ty + 1$,
\end{center}
intersect at the point $(t^{-1}-1,-2t^{-1}+1)$. The tropicalization of this point is $(-1,-1)$, and so if $I = (f,g)$ then
\begin{center}
  $\mathcal{T}(\textbf{V}(I)) = (-1,-1)$.
\end{center}
However, is we tropicalize the linear equations we get:
\begin{center}
  $\mathcal{T}(f) = X \oplus Y \oplus 0$, \hspace{.1 in} and \hspace{.1 in} $\mathcal{T}(g) = 1X \oplus 1Y \oplus 0$.
\end{center}
Each of $\textbf{V}(\mathcal{T}(f))$ and $\textbf{V}(\mathcal{T}(g))$ is a tropical line, and their intersection is the tropical prevariety consisting of all points $(a,a)$ with $a \leq -1$.

\vspace{.1 in}
\begin{tabular}{c}
  \centering
  \hspace{1in}\includegraphics[scale=1]{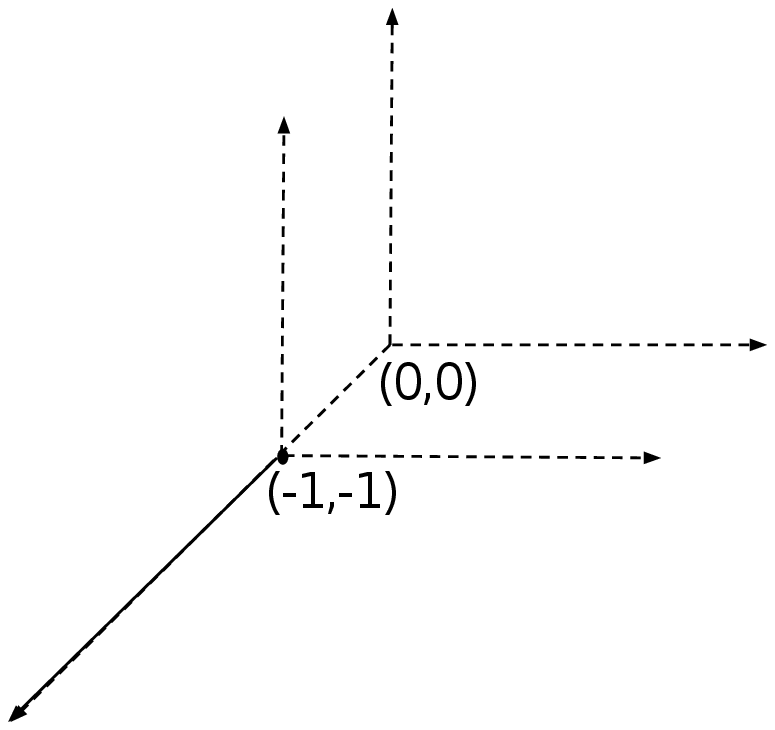}
\end{tabular}

This tropical prevariety properly contains the tropical variety $(-1,-1)$, but the prevariety is clearly much larger. That the intersection of two distinct tropical lines is not necessarily a point is a motivating example of why we do not define a tropical variety to be a finite intersection of tropical hypersurfaces.

\subsection{Kapranov and tropical Rank}

In \cite{dss}, Develin, Santos, and Sturmfels define three notions of matrix rank coming from tropical geometry: the Barvinok rank, the Kapranov rank, and the tropical rank. In this paper we focus on the symmetric analogs of the Kapranov and tropical ranks.

When denoting the element in row $i$ and column $j$ of a matrix these indices will be separated by a comma, so for example $A_{i,j}$ is element $(i,j)$ of the matrix $A$.

The \emph{tropical rank} of an $m \times n$ matrix $A = (A_{i,j}) \in \mathbb{R}^{m \times n}$ is the smallest number $r \leq min(m,n)$ such that $A$ is not in the tropical prevariety formed by the $r \times r$ minors of an $m \times n$ matrix of indeterminates.

A \emph{lift} of the matrix $A$ is a matrix $\tilde{A} = (\tilde{a}_{i,j}) \in (\tilde{K}^{*})^{m \times n}$ such that $deg(\tilde{a}_{i,j}) = A_{i,j}$ for all $i,j$. The \emph{Kapranov rank} of a matrix is the smallest rank of any lift of the matrix. Equivalently, the Kapranov rank is the smallest number $r \leq min(m,n)$ such that $A$ is not in the tropical variety formed by the $r \times r$ minors of an $m \times n$ matrix of indeterminates.

The tropical variety defined by the $r \times r$ minors of an $m \times n$ matrix of indeterminates is contained in the tropical prevariety defined by the same minors, and therefore
\begin{center}
  tropical rank $\leq$ Kapranov rank.
\end{center}

In this paper we define symmetric analogs of the Kapranov and tropical ranks. When we want to emphasize we're referring to the Kapranov and tropical ranks defined here, we will sometimes refer to them as standard Kapranov rank and standard tropical rank.

\subsection{When do the minors of a matrix form a tropical basis?}

A natural question to ask about Kapranov and tropical rank are when they are necessarily equal. In other words, for what values $r,m,n$ does tropical rank $r$ imply Kapranov rank $r$ for any $m \times n$ matrix.

This question was answered through the combined work of Develin, Santos, and Sturmfels \cite{dss}, Chan, Jensen, and Rubei \cite{cjr}, and Shitov \cite{sh2}. The result is named after Shitov \cite{ms}, as he completed the project.

\begin{thm}[Shitov's Theorem]
  The $r \times r$ minors of an $m \times n$ matrix of indeterminates form a tropical basis if and only if:
  \begin{itemize}
    \item $r = min(m,n)$, or
    \item $r \leq 3$, or
    \item $r = 4$ and $min(m,n) \leq 6$.
  \end{itemize}
\end{thm}

The main result of this paper is a partial analog of this result for symmetric matrices.

\section{Symmetric Kapranov and symmetric tropical rank}

The symmetric Kapranov and symmetric tropcial ranks are defined analogously to their general counterparts.

The \emph{symmetric tropical rank} of an $n \times n$ symmetric matrix $A = (A_{i,j}) \in \mathbb{R}^{n \times n}$ is the smallest number $r \leq n$ such that $A$ is not in the tropical prevariety formed by the $r \times r$ minors of an $n \times n$ symmetric matrix of indeterminates.

The \emph{symmetric Kapranov rank} of an $n \times n$ symmetric matrix $A = (A_{i,j}) \in \mathbb{R}^{n \times n}$ is the smallest rank of any lift to a symmetric matrix. Equivalently, the symmetric Kapranov rank is the smallest number $r \leq n$ such that $A$ is not in the tropical variety formed by the $r \times r$ minors of an $n \times n$ symmetric matrix of indeterminates.

The tropical variety defined by the $r \times r$ minors of an $n \times n$ symmetric matrix of indeterminates is contained in the tropical prevariety defined by the same minors, and therefore
\begin{center}
  symmetric tropical rank $\leq$ symmetric Kapranov rank.
\end{center}

As a lift to an $n \times n$ symmetric matrix is a lift to an $n \times n$ matrix, we must have
\begin{center}
  Kapranov rank $\leq$ symmetric Kapranov rank.
\end{center}
In the next subsection, we see the symmetric Kapranov rank can be greater than the standard Kapranov rank, and investigate this in the context of tropical conics.

\subsection{Singular tropical conics}

In classical algebraic geometry a quadric is a hypersurface in $\mathbb{P}^{n-1}$ defined by a homogeneous polynomial $f \in k[x_{1},\ldots,x_{n}]$ of degree two
\begin{center}
  $f = a_{11}x_{1}^{2} + a_{12}x_{1}x_{2} + \cdots + a_{1n}x_{1}x_{n} + a_{22}x_{2}^{2} + a_{23}x_{2}x_{3} + \cdots + a_{nn}x_{n}^{2}$.
\end{center}
For each such polynomial there is a corresponding symmetric matrix, $A$, defined by the relations
\begin{center}
  $f = \left(\begin{array}{cccc} x_{1} & x_{2} & \cdots & x_{n}\end{array}\right)\left(\begin{array}{cccc} a_{11} & \frac{1}{2}a_{12} & \cdots & \frac{1}{2}a_{1n} \\ \frac{1}{2}a_{12} & a_{22} & \cdots & \frac{1}{2}a_{2n} \\ \vdots & \vdots & \ddots & \vdots \\ \frac{1}{2}a_{1n} & \frac{1}{2}a_{2n} & \cdots & a_{nn} \end{array}\right)\left(\begin{array}{c} x_{1} \\ x_{2} \\ \vdots \\ x_{n} \end{array}\right) = \textbf{x}^{T}A\textbf{x}$.
\end{center}
The hypersurface $\textbf{V}(f)$ is singular if and only if the corresponding symmetric matrix is singular, and the rank of a quadric is defined to be the rank of the corresponding symmetric matrix.

Quadrics in $\mathbb{P}^{2}$ are called \emph{conics}, and it's a standard result that a singular conic is the union of two lines. If we rename our variables $x,y,z$, and restrict to the affine subspace given by $z = 1$, we can define a conic to be the hypersurface \textbf{V}(g) of the polynomial
\begin{center}
  $g = ax^{2}+bxy+cy^{2}+dx+ey+f$
\end{center}
with corresponding symmetric matrix
\begin{center}
  $f = \left(\begin{array}{ccc} x & y & 1\end{array}\right)\left(\begin{array}{ccc} a & \frac{1}{2}b & \frac{1}{2}d \\ \frac{1}{2}b & c & \frac{1}{2}e \\ \frac{1}{2}d & \frac{1}{2}e & f \end{array}\right)\left(\begin{array}{c} x \\ y \\ 1 \end{array}\right)$.
\end{center}

The tropicalization of $g$ is the tropical polynomial
\begin{center}
  $G = AX^{2} \oplus BXY \oplus CY^{2} \oplus DX \oplus EY \oplus F$
\end{center}
where the coefficients, variables, and operations are, of course, all replaced by their tropical counterparts. The tropical hypersurface \textbf{V}(G) is a \emph{tropical conic}. The corresponding symmetric matrix is
\begin{center}
  $\left(\begin{array}{ccc} A & B & D \\ B & C & E \\ D & E & F \end{array}\right)$
\end{center}

Now, let's investigate two different tropical conics. The first is the tropical conic defined by
\begin{center}
  $G_{1} = 1X^{2} \oplus XY \oplus 1Y^{2} \oplus X \oplus Y \oplus 0$,
\end{center}
which has corresponding symmetric matrix
\begin{center}
  $C_{1} = \left(\begin{array}{ccc} 1 & 0 & 0 \\ 0 & 1 & 0 \\ 0 & 0 & 0 \end{array}\right)$.
\end{center}
This matrix has less than full Kapranov rank, and indeed we can easily find a lift to a singular matrix over $\tilde{K}$
\begin{center}
  $\left(\begin{array}{ccc} t & 1 & 1+t \\ 1 & t & 1+t \\ 1+t & 1+t & 2+2t \end{array}\right)$.
\end{center}
The tropical conic $\textbf{V}(G_{1})$ is the union of two tropical lines, and so it makes sense to view this as a singular tropical conic.

\vspace{.1 in}
\begin{tabular}{c}
  \centering
  \hspace{1in}\includegraphics[scale=1]{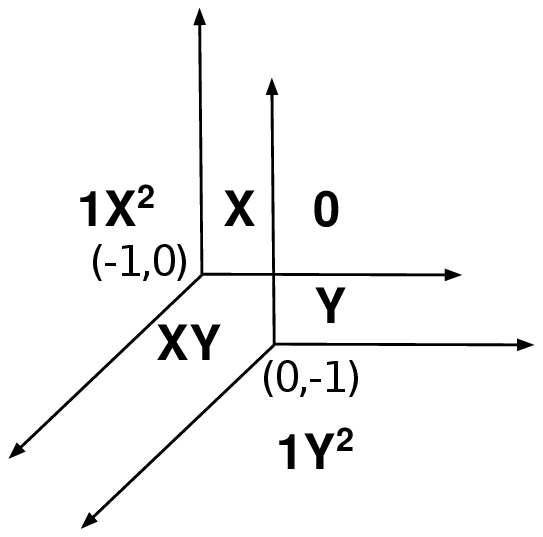}
\end{tabular}

On the other hand, the tropical conic defined by the tropical polynomial
\begin{center}
  $G_{2} = 1X^{2} \oplus XY \oplus 1Y^{2} \oplus X \oplus Y \oplus 0$
\end{center}
has corresponding symmetric matrix
\begin{center}
  $C_{2} = \left(\begin{array}{ccc} 1 & 0 & 0 \\ 0 & 1 & 0 \\ 0 & 0 & 1 \end{array}\right)$.
\end{center}
This matrix also has less than full Kapranov rank, and we can find a lift to a singular matrix over $\tilde{K}$
\begin{center}
  $\left(\begin{array}{ccc} t & 1 & 1+t \\ 1 & t & 1+t \\ 1+t & -1 & t \end{array}\right)$.
\end{center}
However, the tropical conic $\textbf{V}(G_{2})$ is clearly \emph{not} the union of two tropical lines, and so it would be odd to call it singular.

\vspace{.1 in}
\begin{tabular}{c}
  \centering
  \hspace{1in}\includegraphics[scale=1]{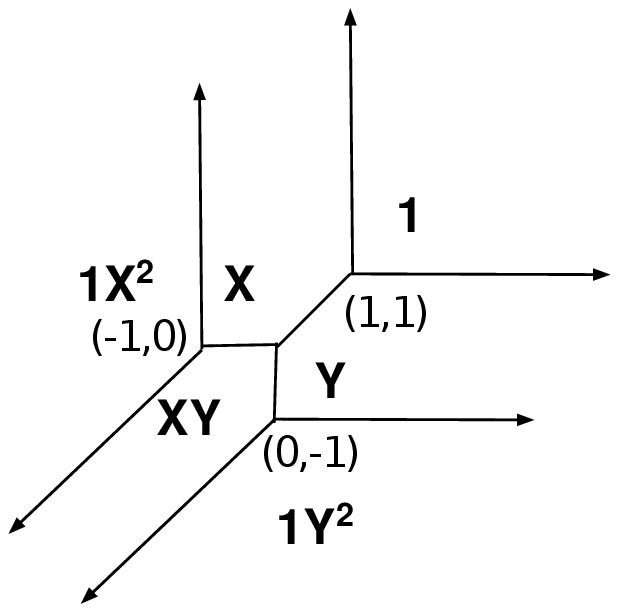}
\end{tabular}

The critical distinction here is the lift of $C_{1}$ is symmetric, while the lift of $C_{2}$ is not, and that's not just because of our particular choice of lifts. It's impossible to find a singular, symmetric lift of $C_{2}$, and so while it has less than full Kapranov rank, it has full symmetric Kapranov rank, and $C_{2}$ is an example of a symmetric matrix with symmetric Kapranov rank greater than its Kapranov rank. 

To prove this about $C_{2}$, for the sake of contradiction suppose $C_{2}$ has a lift to a symmetric rank two matrix 
\begin{center}
  $\tilde{C}_{2} := \left(\begin{array}{ccc}c_{1,1}t + \cdots & c_{1,2} + \cdots & c_{1,3} + \cdots \\ c_{1,2} + \cdots & c_{2,2}t + \cdots & c_{2,3} + \cdots \\ c_{1,3} + \cdots & c_{2,3} + \cdots & c_{3,3}t + \cdots \end{array}\right)$,
\end{center}
where $c_{i,j} \in \mathbb{C}$.
If the first column of $\tilde{C}_{2}$ were a $\tilde{K}$-multiple of the second,
\begin{center}
  $\tilde{\textbf{c}_{1}} = \tilde{k}\tilde{\textbf{c}}_{2}$,
\end{center}
then the relation from the first row
\begin{center}
  $c_{1,1}t + \cdots = \tilde{k}(c_{1,2} + \cdots)$
\end{center}
would require $deg(\tilde{k}) = 1$, while the relation from the second row
\begin{center}
  $c_{1,2} + \cdots = \tilde{k}(c_{2,2}t + \cdots)$
\end{center}
would require $deg(\tilde{k}) = -1$. This cannot be, and so the second column of $\tilde{C}_{2}$ is linearly independent of the first.
As the first two columns are linearly independent, if $\tilde{C}_{2}$ has rank two there must be a linear combination of the first two columns equal to the third
\begin{center}
  $\tilde{k}_{1}\tilde{\textbf{c}}_{1} + \tilde{k}_{2}\tilde{\textbf{c}}_{2} = \tilde{\textbf{c}}_{3}$.
\end{center}
Explicitly, this is the three equalities:
\begin{center}
  $\tilde{k}_{1}(c_{1,1}t + \cdots) + \tilde{k}_{2}(c_{1,2} + \cdots) = c_{1,3} + \cdots$;  

  $\tilde{k}_{1}(c_{1,2} + \cdots) + \tilde{k}_{2}(c_{2,2}t + \cdots) = c_{2,3} + \cdots$;
 
  $\tilde{k}_{1}(c_{1,3} + \cdots) + \tilde{k}_{2}(c_{2,3} + \cdots) = c_{3,3}t + \cdots$.
\end{center}
If $deg(\tilde{k}_{2}) < deg(\tilde{k}_{1})$ then from the first equality we must have $deg(\tilde{k}_{2}) = 0$, but this would make the third equality impossible. If $deg(\tilde{k}_{1}) < deg(\tilde{k}_{2})$ then from the second equality we must have $deg(\tilde{k}_{1}) = 0$, but this would also make the third equality impossible. If $deg(\tilde{k}_{1}) = deg(\tilde{k}_{2})$ then from the first equality (or the second) we must have $deg(\tilde{k}_{1}) = deg(\tilde{k}_{2}) = 0$.
Suppose $deg(\tilde{k}_{1}) = deg(\tilde{k}_{2}) = 0$, and denote the leading terms of $\tilde{k}_{1}$ and $\tilde{k}_{2}$ by, respectively, $k_{1}$ and $k_{2}$. Then the first, second, and third equalities above, respectively, require:
\begin{center}
  $k_{2}c_{1,2} = c_{1,3}$;

  $k_{1}c_{1,2} = c_{2,3}$;

  $k_{1}c_{1,3} = -k_{2}c_{2,3}$.
\end{center}
Substituting the first of these equalities into the left side of the third, and the second into the right side of the third, we derive the equality
\begin{center}
  $k_{1}k_{2}c_{1,2} = -k_{1}k_{2}c_{1,2}$.
\end{center}
This cannot be as neither $k_{1},k_{2}$, nor $c_{1,2}$ is $0$. So, $C_{2}$ has no rank two symmetric lift, and its symmetric Kapranov rank is three.

Generally, a tropical conic will be the union of two tropical lines if and only if its corresponding symmetric matrix has less than full symmetric Kapranov rank, and so the symmetric ranks are the ones that should be used to determine whether a tropical conic is singular. This generalizes to tropical quadrics, although we won't investigate those here.

\subsection{Basic properties}

A square matrix $A = (A_{i,j}) \in \mathbb{R}^{n \times n}$ is \emph{tropically singular} if the minimum
  \begin{center}
    $tropdet(A) := \bigoplus_{\sigma \in S_{n}} A_{1,\sigma(1)} \odot A_{2,\sigma(2)} \odot \cdots \odot A_{d,\sigma(n)}$
  \end{center}
  is attained at least twice is the tropical sum. Here $S_{n}$ denotes the symmetric group on $\{1,2,\ldots,n\}$. We call the number $tropdet(A)$ defined above the \emph{tropical determinant} of $A$, and we say any permutation $\sigma$ such that
  \begin{center}
    $tropdet(A) = A_{1,\sigma(1)} \odot A_{2,\sigma(2)} \odot \cdots \odot A_{d,\sigma(d)}$
  \end{center}
  \emph{realizes} the tropical determinant. So, equivalently, a square matrix $A$ is tropically singular if more than one permutation realizes the tropical determinant.

More generally, suppose $A$ is an $m \times n$ real matrix, and $\{i_{1},i_{2},\ldots,i_{r}\}$ and $\{j_{1},j_{2},\ldots,j_{r}\}$ are subsets of $\{1,\ldots,m\}$ and $\{1,\ldots,n\}$, respectively. These subsets define an $r \times r$ submatrix $A'$ of $A$, with row indices $\{i_{1},\ldots,i_{r}\}$ and column indices $\{j_{1},\ldots,j_{r}\}$. A tropical monomial of the form
\begin{center}
  $\bigodot_{k = 1}^{r}X_{i_{k},\rho(i_{k})}$,
\end{center}
where $\rho$ is a bijection from the row indices to the column indices, is a \emph{minimizing monomial} for the submatrix $A'$ if, over all monomials defined by bijections from $\{i_{1},i_{2},\ldots,i_{r}\}$ to $\{j_{1},j_{2},\ldots,j_{r}\}$, this monomial is minimal under the valuation $X_{i,j} \mapsto A_{i,j}$. An $r \times r$ submatrix of $A$ is tropically singular if it has more than one minimizing monomial.

For symmetric matrices, we say a submatrix is \emph{symmetrically tropically singular} if it has more than one minimizing monomial given the equivalence $X_{i,j} = X_{j,i}$. 

For example, the tropical determinant of a $3 \times 3$ matrix of indeterminates 
\begin{center}
  $\left(\begin{array}{ccc} X_{1,1} & X_{1,2} & X_{1,3} \\ X_{2,1} & X_{2,2} & X_{2,3} \\ X_{3,1} & X_{3,2} & X_{3,3} \end{array}\right)$
\end{center}
is
\begin{center}
  $X_{1,1}X_{2,2}X_{3,3} \oplus X_{1,2}X_{2,3}X_{3,1} \oplus X_{1,3}X_{2,1}X_{3,2} \oplus X_{1,1}X_{2,3}X_{3,2} \oplus X_{1,2}X_{2,1}X_{3,3} \oplus X_{1,3}X_{2,2}X_{3,1}$.
\end{center}
For the matrix $C_{2}$ from the previous subsection, there are two minimizing monomials $X_{1,2}X_{2,3}X_{3,1}$ and $X_{1,3}X_{2,1}X_{3,2}$, and so the matrix is tropically singular. However, under the equivalence $X_{i,j} = X_{j,i}$ the tropical determinant becomes
\begin{center}
  $X_{1,1}X_{2,2}X_{3,3} \oplus X_{1,2}X_{2,3}X_{1,3} \oplus X_{1,1}X_{2,3}^{2} \oplus X_{1,2}^{2}X_{3,3} \oplus X_{1,3}^{2}X_{2,2}$,
\end{center}
and for the matrix $C_{2}$ the monomial $X_{1,2}X_{2,3}X_{1,3}$ is the unique minimizing monomial. Therefore, $C_{2}$ is not symmetrically tropically singular.

The tropical rank of a matrix can be equivalently defined as the smallest value of $r$ such that the matrix has a tropically nonsingular $r \times r$ submatrix, and similarly the symmetric tropical rank of a symmetric matrix can be equivalently defined as the smallest value of $r$ such that the symmetric matrix has a symmetrically tropically nonsingular $r \times r$ submatrix. 

If $S_{n}$ is the set of permutations of $n$ elements, we define an equivalence class on the elements of $S_{n}$ by declaring two permutations to be in the same class if they have the same disjoint cycle decomposition up to inversion of the cycles. In other words, if $\sigma$ is a permutation with disjoint cycle decomposition:
\begin{center}
  $\sigma = \sigma_{1}\sigma_{2} \cdots \sigma_{k}$,
\end{center}
where the $\sigma_{i}$ are disjoint cycles, then the other elements in its equivalence class are of the form:
\begin{center}
  $\sigma_{1}^{\pm}\sigma_{2}^{\pm} \cdots \sigma_{k}^{\pm}$.
\end{center}
For example, in $S_{7}$ the permutations $(1257)(346)$, $(1752)(346)$, $(1257)(364)$, $(1752)(364)$ would all be in the same equivalence class.

Note that as the parity of a permutation is determined completely by the sizes of the cycles in its disjoint cycle decomposition, and as a cycle and its inverse have the same size, every element in a given equivalence class has the same parity.

Denote by $\tilde{S}_{n}$ this equivalence class of permutations in $S_{n}$. If two permutations are in the same equivalence class they are \emph{cycle-similar}, and if not they are \emph{cycle-distinct}. Denote the equivalence class containing the permutation $\sigma$ by $\tilde{\sigma}$.

\newtheorem{prop}{Proposition}
\begin{prop} 
  A symmetric matrix is symmetrically tropically singular if and only if it has two cycle-distinct permutations that realize the determinant.
\end{prop}

\begin{proof}
  Consider the symmetric matrix of variables:
  \begin{center}
    $X := \left(\begin{array}{ccccc} x_{1,1} & x_{1,2} & x_{1,3} & \cdots & x_{1,n} \\ x_{1,2} & x_{2,2} & x_{2,3} & \cdots & x_{2,n} \\ x_{1,3} & x_{2,3} & x_{3,3} & \cdots & x_{3,n} \\ \vdots & \vdots & \vdots & \ddots & \vdots \\ x_{1,n} & x_{2,n} & x_{3,n} & \cdots & x_{n,n} \end{array}\right)$.
  \end{center} 
  For any cycle
  \begin{center}
    $\sigma = (k_{1}k_{2} \cdots k_{m})$
  \end{center}
  define the monomial
  \begin{center}
    $x_{\sigma} = x_{k_{1},k_{2}}x_{k_{2},k_{3}} \cdots x_{k_{m},k_{1}}$,
  \end{center}
  and for any permutation $\tau \in S_{n}$ with disjoint cycle decomposition 
  \begin{center}
    $\tau = \sigma_{1}\sigma_{2} \cdots \sigma_{k}$
  \end{center}
  define the monomial
  \begin{center}
    $x_{\tau} = \prod_{i= 1}^{n} x_{i,\tau(i)} = \prod_{i = 1}^{k}x_{\sigma_{i}}$.
  \end{center}
  We have
  \begin{center}
    $x_{\sigma} = x_{k_{1},k_{2}}x_{k_{2},k_{3}} \cdots x_{k_{m},k_{1}}$, \hspace{.1 in} and \hspace{.1 in} $x_{\sigma^{-1}} = x_{k_{1},k_{m}} \cdots k_{k_{3},k_{2}}x_{k_{2},k_{1}}$.
  \end{center}
  As $x_{i,j} = x_{j,i}$ we see $x_{\sigma} = x_{\sigma^{-1}}$, and therefore for any two cycle-similar permutations $\tau_{1}$ and $\tau_{2}$ we must have $x_{\tau_{1}} = x_{\tau_{2}}$. In other words, the permutations $\tau_{1}$ and $\tau_{2}$ produce the same monomial in the determinant of $X$. Note that as $\tau_{1}$ and $\tau_{2}$ have the same parity the monomials $x_{\tau_{1}}$ and $x_{\tau_{2}}$ have the same sign in the determinant, and there is no concern about identical monomials cancelling. 
  
  On the other hand, suppose for two distinct permutations $\tau_{1}$ and $\tau_{2}$ that, given $x_{i,j} = x_{j,i}$, we have $x_{\tau_{1}} = x_{\tau_{2}}$. The permutation $\tau_{1}$ will have some disjoint cycle decomposition
  \begin{center}
    $\tau_{1} = \sigma_{1}\sigma_{2} \cdots \sigma_{t}$.
  \end{center}
  Suppose
  \begin{center}
    $\sigma_{1} = (k_{1}k_{2} \cdots k_{m})$.
  \end{center}
  This means the variables
  \begin{center}
    $x_{k_{1},k_{2}}x_{k_{2},k_{3}} \cdots x_{k_{m},k_{1}}$
  \end{center}
  appear in the product of variables defining the monomial $x_{\tau_{1}}$. If every one of these variables also appear in $x_{\tau_{2}}$, then the cycle $\sigma_{1}$ also appears in the disjoint cycle decomposition of $\tau_{2}$. If this is the case for every cycle in the cycle decomposition of $\tau_{1}$, then $\tau_{1} = \tau_{2}$. 
  
  So, assume without loss of generality that $\sigma_{1}$ is not in the disjoint cycle decomposition of $\tau_{2}$, and the variable $x_{k_{1},k_{2}}$ does not appear in $x_{\tau_{2}}$. As the only relation between the variables is $x_{i,j} = x_{j,i}$, if $x_{k_{1},k_{2}}$ does not appear in $x_{\tau_{2}}$, then $x_{k_{2},k_{1}}$ must. This means $x_{k_{2},k_{3}}$ cannot appear in $x_{\tau_{2}}$, and so $x_{k_{3},k_{2}}$ must. Repeating this argument we see that the product of variables
  \begin{center}
    $x_{k_{2},k_{1}}x_{k_{3},k_{2}} \cdots x_{k_{m},k_{m-1}}x_{k_{1},k_{m}}$
  \end{center}
  must appear in $x_{\tau_{2}}$, which means $\tau_{2}$ must contain in its disjoint cycle decomposition the cycle
  \begin{center}
    $(k_{m}k_{m-1} \cdots k_{1}) = (k_{1}k_{2} \cdots k_{s})^{-1}$.
  \end{center}
  So, for every cycle in the disjoint cycle decomposition of $\tau_{1}$ either that cycle or its inverse appears in $\tau_{2}$, and obviously vice-versa. Ergo, $\tau_{1}$ and $\tau_{2}$ are cycle-similar. 
  
  From this we conclude the distinct monomials occuring in the determinant of $X$ are the cycle-distinct monomials, and therefore a symmetric matrix is symmetrically tropically singular if and only if it has two cycle-distinct permutations that realize the determinant.
\end{proof}

If an $r \times r$ submatrix of a symmetric $n \times n$ matrix has two distinct minimizing monomials given the equivalence $X_{i,j} = X_{j,i}$ then a fortiori it has two distinct minimizing monomials without that equivalence, and so
\begin{center}
  tropical rank $\leq$ symmetric tropical rank.
\end{center}
The matrix $C_{2}$ is an example of a matrix with tropical rank less than its symmetric tropical rank.

One situation where tropical rank and symmetric tropical rank are necessarily equal is when both are one.

\begin{prop}
  A symmetric matrix $A$ has tropical rank one if and only if it has symmetric tropical rank one.
\end{prop}

\begin{proof}
  Rank one is the minimum possible rank. As tropical rank cannot be greater than symmetric tropical rank, symmetric tropical rank one implies tropical rank one.
  
  The determinant of a $2 \times 2$ submatrix of a symmetric $n \times n$ matrix of variables is the difference of two monomials, the product of the diagonal terms, and the product of the off-diagonal terms, and these monomials cannot be the same even under the equivalence $X_{i,j} = X_{j,i}$. If a matrix has tropical rank one, then for every $2 \times 2$ submatrix the sum of the diagonal terms equals the sum of the off-diagonal terms. This means every $2 \times 2$ submatrix is symmetrically tropically singular, and the matrix has symmetric tropical rank one. 
\end{proof}

\newtheorem{cor}{Corollary}
\begin{cor}
  If a symmetric matrix has symmetric tropical rank two then it has tropical rank two.
\end{cor}

\begin{proof}
  The tropical rank cannot be greater the symmetric tropical rank, and by the above proposition if the tropical rank were one, the symmetric tropical rank would be one as well. So, the tropical rank must be two.
\end{proof}

The matrix $C_{2}$ has tropical rank two but greater symmetric tropical rank. The form of $C_{2}$ is, essentially, the only way this is possible.

\begin{prop}
  A real symmetric matrix of tropical rank two has greater symmetric tropical rank if and only if a principal $3 \times 3$ submatrix is not symmetrically tropically singular.
\end{prop}

\begin{proof}
  If any $3 \times 3$ submatrix of a real symmetric matrix is not symmetrically tropically singular, then the matrix has symmetric tropical rank greater than two. So, what must be proven is that for a real symmetric matrix if a $3 \times 3$ submatrix is not a principal submatrix then tropically singular implies symmetrically tropically singular.
  
  Take any $3 \times 3$ submatrix from an $n \times n$ symmetric matrix of variables
  \begin{center}
    $\left(\begin{array}{ccc} x_{i,p} & x_{i,q} & x_{i,r} \\ x_{j,p} & x_{j,q} & x_{j,r} \\ x_{k,p} & x_{k,q} & x_{k,r} \end{array}\right)$,
  \end{center}
  where $i < j < k$, and $p < q < r$. The determinant of this submatrix is the polynomial
  \begin{center}
    $x_{i,p}x_{j,q}x_{k,r} + x_{i,q}x_{j,r}x_{k,p} + x_{i,r}x_{j,p}x_{k,q} - x_{i,p}x_{j,r}x_{k,q} - x_{i,q}x_{j,p}x_{k,r} - x_{i,r}x_{j,q}x_{k,p}$.
  \end{center}
  Suppose, given the symmetry of the $n \times n$ matrix of variables, that two of these monomials are equal. If $i < p$ then $i$ is not the index of any column in our submatrix, and symmetry provides no duplication of variables from row $i$. This means if a monomial in the $3 \times 3$ determinant above is duplicated, the monomials in a $2 \times 2$ minor are duplicated. This is impossible. Identical logic applies if $p < i$, and therefore $i = p$. Applying the same argument we get $j = q$ and $k = r$. So, the only situation where tropically singular and symmetrically tropically singular can differ for a $3 \times 3$ submatrix is if that submatrix is principal.
\end{proof}

In standard linear algebra if one column (or row) of a square matrix is a multiple of another, then that matrix must be singular. The same is true for symmetric tropical matrices.

\begin{prop}
  If $A$ is an $r \times r$ submatrix of an $n \times n$ symmetric matrix, and one row of $A$ is a tropical multiple of another, then $A$ is symmetrically tropically singular. The same is true if one column of $A$ is a tropical multiple of another.
\end{prop}

\begin{proof}
  Suppose $A$ is formed from the row indices $i_{1},\ldots,i_{r}$ and the column indices $j_{1},\ldots,j_{r}$ of the $n \times n$ symmetric matrix. Denote the rows of $A$ by $\textbf{a}_{i_{1}},\textbf{a}_{i_{2}},\ldots,\textbf{a}_{i_{r}}$. We may assume without loss of generality that $\textbf{a}_{i_{r}} = c \odot \textbf{a}_{i_{r-1}}$, where $c \in \mathbb{R}$. Suppose the monomial
  \begin{center}
    $X_{1} = X_{i_{1},\rho(i_{1})} \odot X_{i_{2},\rho(i_{2})} \odot \cdots \odot X_{i_{r-1},\rho(i_{r-1})} \odot X_{i_{r},\rho(i_{r})}$,
  \end{center}
  where $\rho$ is a bijection from the column indices of $A$ to the row indices, is a minimizing monomial for $A$. Given the equivalence of $\textbf{a}_{i_{r}}$ and $c \odot \textbf{a}_{i_{r-1}}$ the monomial
  \begin{center}
    $X_{2} = X_{i_{1},\rho(i_{1})} \odot X_{i_{2},\rho(i_{2})} \odot \cdots \odot X_{i_{r-1},\rho(i_{r})} \odot X_{i_{r},\rho(i_{r-1})}$
  \end{center}
  must have the same valuation as $X_{1}$, and therefore also be a minimizing monomial. If $X_{1} = X_{2}$ under the equivalence $X_{i,j} = X_{j,i}$ then this would require one of the four equalities below to be true:
  \begin{center}
    $X_{i_{r-1},\rho(i_{r-1})} = X_{i_{r-1},\rho(i_{r})}$; \hspace{.1 in} $X_{i_{r-1},\rho(i_{r-1})} = X_{\rho(i_{r}),i_{r-1}}$;
    
    $X_{i_{r-1},\rho(i_{r-1})} = X_{i_{r},\rho(i_{r-1})}$; \hspace{.1 in} $X_{i_{r-1},\rho(i_{r-1})} = X_{\rho(i_{r-1}),i_{r}}$.
  \end{center}
  Given $i_{r-1} \neq i_{r}$ and $\rho$ is a bijection, none of these equalities is possible. So, even under the equivalence $X_{i,j} = X_{j,i}$, the minimizing monomials $X_{1}$ and $X_{2}$ are distinct, and therefore $A$ is symmetrically tropically singular.
  
  An identical proof applies if one column is a tropical multiple of another.
\end{proof}

\section{When the minors of a symmetric matrix form a tropical basis}

In this section we examine all cases where the $r \times r$ minors of an $n \times n$ symmetric matrix of variables form a tropical basis, with the exception of the boundary case $r = 4$. These cases are $r = 2$, $r = 3$, and $r = n$.

To prove this, we will want a couple useful facts:

\begin{itemize}
  
\item If $A$ is a symmetric matrix, and we permute the rows of $A$ by a permutation $\sigma$, and the columns of $A$ by the \emph{same} permutation, then the resulting matrix $A'$ will be symmetric, and $A'$ will have the same symmetric tropical and symmetric Kapranov rank as $A$. We call a permutation of the rows and columns of $A$ by the same permutation a \emph{diagonal permutation}. 
  
\item If $A$ is a symmetric matrix, and we tropically multiply row $i$ by a constant $c$, and tropically multiply column $i$ by the \emph{same} constant, then the resulting matrix $A'$ will be symmetric, and $A'$ will have the same symmetric tropical and symmetric Kapranov rank as $A$. We call such an operation a \emph{symmetric} scaling of $A$.
  
\end{itemize}

In particular, we will assume without loss of generality that any symmetric matrix $A$ has been symmetrically scaled so that every row and column has $0$ as its minimal entry.

\subsection{Singular symmetric matrices}

By definition, a symmetric matrix is singular if it has less than full rank, and it is a fundamental result in linear algebra that this is the case if and only if the matrix has zero determinant.

\begin{thm}
  The determinant of a symmetric matrix of variables is a tropical basis for the ideal it generates. Equivalently, the $n \times n$ minor of an $n \times n$ symmetric matrix of variables forms a tropical basis.
\end{thm}

\begin{proof}
  The determinant of a symmetric matrix of variables is a single polynomial, and is therefore a tropical basis by Kapranov's theorem.
\end{proof}

\subsection{Rank one symmetric matrices}

The rank one case is straightforward.

\begin{thm}
  A symmetric matrix has symmetric tropical rank one if and only if it has symmetric Kapranov rank one. Equivalently, the $2 \times 2$ minors of a symmetric matrix of variables are a tropical basis.
\end{thm}

\begin{proof}
  As the symmetric tropical rank cannot be greater than the symmetric Kapranov rank, any symmetric matrix with symmetric Kapranov rank one must also have symmetric tropical rank one.
  
  If a symmetric matrix has symmetric tropical rank one, then by Proposition 1 it also has standard tropical rank one. This means every column of the matrix is a constant tropical multiple of the first column. If our matrix is of the form:
  \begin{center}
    $A = \left(\begin{array}{cccc} a_{1,1} & a_{1,2} & \cdots & a_{1,n} \\ a_{2,1} & a_{2,2} & \cdots & a_{2,n} \\ \vdots & \vdots & \ddots & \vdots \\ a_{n,1} & a_{n,2} & \cdots & a_{n,n} \end{array}\right)$,
  \end{center}
  and $\textbf{a}_{i}$ represents column $i$ of the matrix $A$, then $\textbf{a}_{i} = c_{i} \odot \textbf{a}_{1}$ for some constant $c_{i}$. By assumption $A$ is symmetric, so $a_{i,j} = a_{j,i}$. The matrix $A$ is the tropicalization of the matrix
  \begin{center}
    $\tilde{A} = \left(\begin{array}{cccc} \tilde{a}_{1,1} & \tilde{a}_{1,2} & \cdots & \tilde{a}_{1,n} \\ \tilde{a}_{2,1} & \tilde{a}_{2,2} & \cdots & \tilde{a}_{2,n} \\ \vdots & \vdots & \ddots & \vdots \\ \tilde{a}_{n,1} & \tilde{a}_{n,2} & \cdots & \tilde{a}_{n,n} \end{array}\right)$,
  \end{center}
  where $\tilde{a}_{i,1} = t^{a_{i,1}}$, and $\tilde{a}_{i,j} = t^{c_{j}}\tilde{a}_{i,1}$. The matrix $\tilde{A}$ has rank one by construction, and as $a_{i,j} = a_{j,i}$ we have 
  \begin{center}
    $\tilde{a}_{i,j} = t^{c_{j}}\tilde{a}_{i,1} = t^{c_{j} + a_{i,1}} = t^{a_{i,j}} = t^{a_{j,i}}$
    
    $= t^{c_{i} + a_{j,1}} = t^{c_{i}}t^{a_{j,1}} = t^{c_{i}}\tilde{a_{j,1}} = \tilde{a_{j,i}}$. 
  \end{center}
  So, $\tilde{A}$ is symmetric, and therefore $A$ has Kapranov rank one.
\end{proof}

\begin{cor}
  A $3 \times 3$ symmetric matrix $A$ has symmetric Kapranov rank two if and only if it has symmetric tropical rank two.
\end{cor}

\begin{proof}
  If $A$ has symmetric Kapranov rank two, then its symmetric tropical rank cannot be more than two, and by Theorem 5 its symmetric tropical rank cannot be one.
  
  If $A$ has symmetric tropical rank two its symmetric Kapranov rank must be at least two, and by Theorem 4 its symmetric Kapranov rank cannot be three.
\end{proof}

\subsection{Rank two symmetric matrices}

In this subsection we prove the $3 \times 3$ minors of a symmetric $n \times n$ matrix form a tropical basis. The proof is built on the foundation of several lemmas. In several places the proof uses ideas and modifications of arguments from the corresponding proof in Section 6 of \cite{dss}. A few times we will make the inductive assumption that, for a given natural number $n$, the $3 \times 3$ minors of an $m \times m$ symmetric matrix form a tropical basis if $m < n$. The $n = 3$ base case is Corollary 2. 

\newtheorem{lem}{Lemma}
\begin{lem}
  Suppose $A$ is a matrix of the form
  \begin{center}      
    $\left(\begin{array}{ccc} \textbf{0} & \textbf{0} & C \\ \textbf{0} & 0 & \textbf{0} \\ C^{T} & \textbf{0} & \textbf{0} \end{array}\right)$,
  \end{center}   
  where $C$ is nonnegative and has no zero column. If $A$ has symmetric tropical rank two, it has symmetric Kapranov rank two.
\end{lem}

\begin{proof}
  If $C$ is a $k \times l$ matrix, we number the rows and columns of $A$ from $-k$ to $l$. The upper-left zero matrix is the submatrix of nonpositive indices, and the bottom-right zero matrix is the submatrix of nonnegative indices. Note they both contain the center element $A_{0,0}$. 
  
  As $C$ does not contain a zero column we may, possibly after a diagonal permutation, assume the entries $A_{-1,1} = A_{1,-1}$ are positive.

  We now construct a symmetric rank two lift $\tilde{A}$ of $A$. The upper-right submatrix
  \begin{center}
    $A_{UR} = \left(\begin{array}{cc} \textbf{0} & C \\ 0 & \textbf{0} \end{array}\right)$
  \end{center}  
  has (standard) tropical rank two, and so by Theorem 6.5 from \cite{dss} there exists a rank two lift $\tilde{A}_{UR}$ of this submatrix.\footnote{Theorem 6.5 from \cite{dss} relies upon Corollary 6.4 from the same paper, and Corollary 6.4 contains an error in its proof. A correction for this error is given in the first appendix of \cite{z}.} As $C$ does not contain the zero column, the first two columns of $\tilde{A}_{UR}$ must be linearly independent, and every other column of $\tilde{A}_{UR}$ can be written as a linear combination of these first two columns:
  \begin{center}
    $\lambda_{j}\tilde{\textbf{a}}_{0} + \mu_{j}\tilde{\textbf{a}}_{1} = \tilde{\textbf{a}}_{j}$.
  \end{center}
 
  The relation
  \begin{center}
    $\lambda_{j}\tilde{a}_{0,0} + \mu_{j}\tilde{a}_{0,1} = \tilde{a}_{0,j}$
  \end{center}
  implies the degrees of $\lambda_{j}$ and $\mu_{j}$ cannot both be positive, if one has positive degree the other must have degree zero, and if their degrees are both nonpositive they must be equal. If both $\lambda_{j}$ and $\mu_{j}$ had negative degrees, then given $C$ does not contain the zero column $C$ would have a negative entry, but this is not allowed as $C$ is nonnegative. If $\mu_{j}$ had positive degree then $\lambda_{j}$ would have degree zero, but this cannot be as then $C$ would contain the zero column. So, we must have $deg(\lambda_{j}) \geq deg(\mu_{j}) = 0$. 
  
  We use this lift $\tilde{A}_{UR}$, and its transpose, for the upper-right and bottom-left submatrices of $\tilde{A}$. We must complete the lift with entries $\tilde{a}_{i,j}$ for every $i,j$ with $ij > 0$, such that $deg(\tilde{a}_{i,j}) = 0$, $\tilde{a}_{i,j} = \tilde{a}_{j,i}$, and the entire matrix $\tilde{A}$ has rank two. We begin this task with the $3 \times 3$ central minor:

  \begin{center}
    $\left(\begin{array}{ccc} \tilde{a}_{-1,-1} & \tilde{a}_{-1,0} & \tilde{a}_{-1,1} \\ \tilde{a}_{0,-1} & \tilde{a}_{0,0} & \tilde{a}_{0,1} \\ \tilde{a}_{1,-1} & \tilde{a}_{1,0} & \tilde{a}_{1,1} \end{array}\right)$.
  \end{center}
  
  We pick $\tilde{a}_{1,1}$ such that $deg(\tilde{a}_{1,1}) = 0$, but otherwise generically. We want this matrix to be singular, and so once $\tilde{a}_{1,1}$ has been picked $\tilde{a}_{-1,-1}$ is determined. 

  As $\tilde{a}_{1,1}$ is generic, $\tilde{a}_{-1,-1}$ is as well. If $deg(\tilde{a}_{-1,-1}) < 0$, then in order for the above $3 \times 3$ matrix to be singular the leading terms in $\tilde{a}_{0,0}\tilde{a}_{1,1}-\tilde{a}_{0,1}\tilde{a}_{1,0}$ would need to cancel, which is impossible if $\tilde{a}_{1,1}$ is generic. If $deg(\tilde{a}_{-1,-1}) > 0$, then as $deg(\tilde{a}_{-1,1}) = deg(\tilde{a}_{1,-1}) > 0$ there would only be a single degree zero term, $\tilde{a}_{-1,0}\tilde{a}_{0,-1}\tilde{a}_{1,1}$, in the determinant of the $3 \times 3$ central minor, which would make it nonsingular. So, $deg(\tilde{a}_{-1,-1}) = 0$. 

  From here every term $\tilde{a}_{i,1}$ and $\tilde{a}_{i,-1}$, with $i > 1$ or $i < -1$, respectively, is determined by the need for the matrices
  \begin{center}
    $\left(\begin{array}{ccc} \tilde{a}_{-1,-1} & \tilde{a}_{-1,0} & \tilde{a}_{-1,1} \\ \tilde{a}_{0,-1} & \tilde{a}_{0,0} & \tilde{a}_{0,1} \\ \tilde{a}_{i,-1} & \tilde{a}_{i,0} & \tilde{a}_{i,1} \end{array}\right)$ \hspace{.1 in} and \hspace{.1 in} $\left(\begin{array}{ccc} \tilde{a}_{i,-1} & \tilde{a}_{i,0} & \tilde{a}_{i,1} \\ \tilde{a}_{0,-1} & \tilde{a}_{0,0} & \tilde{a}_{0,1} \\ \tilde{a}_{1,-1} & \tilde{a}_{1,0} & \tilde{a}_{1,1} \end{array}\right)$
  \end{center}
  to be, respectively, singular, and that $\tilde{a}_{1,1}$ and $\tilde{a}_{-1,-1}$ are generic ensures all these terms are also generic and have degree zero. The remaining entries $i,j > 1$ in the bottom-right zero matrix are determined by the relations:
  \begin{center}
    $\lambda_{j}\tilde{a}_{i,0} + \mu_{j}\tilde{a}_{i,1} = \tilde{a}_{i,j}$.
  \end{center}
  As $\tilde{a}_{i,1}$ is generic, $deg(\tilde{a}_{i,j}) = 0$ even if $deg(\lambda_{j}) = deg(\mu_{j})$. The degree zero upper-left entries are determined similarly.
  
  It remains to be proven that our lift is symmetric. We first prove $a_{1,i} = a_{i,1}$, with $i > 1$. We examine the matrices
  
  \begin{center}
    $\left(\begin{array}{ccc} \tilde{a}_{-1,-1} & \tilde{a}_{-1,0} & \tilde{a}_{-1,i} \\ \tilde{a}_{0,-1} & \tilde{a}_{0,0} & \tilde{a}_{0,i} \\ \tilde{a}_{1,-1} & \tilde{a}_{1,0} & \tilde{a}_{1,i} \end{array}\right)$ \hspace{.1 in} and \hspace{.1 in} $\left(\begin{array}{ccc} \tilde{a}_{-1,-1} & \tilde{a}_{-1,0} & \tilde{a}_{-1,1} \\ \tilde{a}_{0,-1} & \tilde{a}_{0,0} & \tilde{a}_{0,1} \\ \tilde{a}_{i,-1} & \tilde{a}_{i,0} & \tilde{a}_{i,1} \end{array}\right)$.
  \end{center}
  
  By construction 
  \begin{center}
    $\tilde{a}_{-1,0} = \tilde{a}_{0,-1}$, \hspace{.1 in} $\tilde{a}_{1,0} = \tilde{a}_{0,1}$, 
    
    $\tilde{a}_{-1,1} = \tilde{a}_{1,-1}$, \hspace{.1 in} and \hspace{.1 in} $\tilde{a}_{-1,i} = \tilde{a}_{i,-1}$. 
  \end{center}
  So, the formula for the determinant of the first matrix is the same as the formula for the determinant of the second with $\tilde{a}_{1,i}$ replaced by $\tilde{a}_{i,1}$. As both matrices are singular we must have $\tilde{a}_{1,i} = \tilde{a}_{i,1}$.
  
  For the remaining terms verifying symmetry is a straightforward calculation (here $i,j > 1$):
  \begin{center}
    $\tilde{a}_{i,j} = \lambda_{j} \tilde{a}_{i,0} + \mu_{j} \tilde{a}_{i,1} = \lambda_{j} \tilde{a}_{0,i} + \mu_{i} \tilde{a}_{1,i}$

    $= \lambda_{j} (\lambda_{i} \tilde{a}_{0,0} + \mu_{i} \tilde{a}_{0,1}) + \mu_{j} (\lambda_{i} \tilde{a}_{1,0} + \mu_{i} \tilde{a}_{1,1})$

    $= \lambda_{i} (\lambda_{j} \tilde{a}_{0,0} + \mu_{j} \tilde{a}_{1,0}) + \mu_{i} (\lambda_{j} \tilde{a}_{0,1} + \mu_{j} \tilde{a}_{1,1})$

    $= \lambda_{i} (\lambda_{j} \tilde{a}_{0,0} + \mu_{j} \tilde{a}_{0,1}) + \mu_{i} (\lambda_{j} \tilde{a}_{1,0} + \mu_{j} \tilde{a}_{1,1})$

    $= \lambda_{i} \tilde{a}_{0,j} + \mu_{i} \tilde{a}_{1,j} = \lambda_{i} \tilde{a}_{j,0} + \mu_{i} \tilde{a}_{j,1} = \tilde{a}_{j,i}$.
    
  \end{center}
  The verification of symmetry for $i,j < -1$ is essentially identical. So, we have constructed a symmetric rank two lift $\tilde{A}$ of $A$, and therefore $A$ has symmetric Kapranov rank two.
\end{proof}

\begin{lem}
  Suppose $A$ is a matrix of the form:  
  \begin{center}  
    $\left(\begin{array}{ccc} B_{1} & \textbf{0} & \textbf{0} \\ \textbf{0} & 0 & \textbf{0} \\ \textbf{0} & \textbf{0} & B_{2}\end{array}\right)$
  \end{center}
  where $B_{1}$ and $B_{2}$ are nonempty, positive symmetric matrices. If $A$ has symmetric tropical rank two, then it has symmetric Kapranov rank two.
\end{lem}

\begin{proof}
  As in the previous lemma we number the rows and columns from $-k$ to $l$, where here $k \times k$ and $l \times l$ are the dimensions of $B_{1}$ and $B_{2}$, respectively.
  
  By induction we may assume the matrices
  \begin{center}
    $\left(\begin{array}{cc} B_{1} & \textbf{0} \\ \textbf{0} & 0 \end{array}\right)$ \hspace{.1 in} and \hspace{.1 in} $\left(\begin{array}{cc} 0 & \textbf{0} \\ \textbf{0} & B_{2} \end{array}\right)$
  \end{center}
  have symmetric rank two lifts $\tilde{B}_{1}$ and $\tilde{B}_{2}$, respectively, and after possibly multipling the left column and top row of $\tilde{B}_{2}$ by the same degree zero constant, we may assume the bottom-right entry of $\tilde{B}_{1}$ is equal to the top-left entry of $\tilde{B}_{2}$.
  
  We now construct a symmetric rank two lift $\tilde{A}$ of $A$. We begin with the lifts $\tilde{B}_{1}$ and $\tilde{B}_{2}$, and construct the entries in the upper-right zero matrix.
  
  Like in Lemma 2 we start with the $3 \times 3$ principal submatrix:
  \begin{center}
    $\left(\begin{array}{ccc} \tilde{a}_{-1,-1} & \tilde{a}_{-1,0} & \tilde{a}_{-1,1} \\ \tilde{a}_{0,-1} & \tilde{a}_{0,0} & \tilde{a}_{0,1} \\ \tilde{a}_{1,-1} & \tilde{a}_{1,0} & \tilde{a}_{1,1} \end{array}\right)$. 
  \end{center}
  We need this matrix to be singular and symmetric. This means we must find $x$ such that
  \begin{center}
    $\left|\begin{array}{ccc} \tilde{a}_{-1,-1} & \tilde{a}_{-1,0} & x \\ \tilde{a}_{0,-1} & \tilde{a}_{0,0} & \tilde{a}_{0,1} \\ x & \tilde{a}_{1,0} & \tilde{a}_{1,1} \end{array}\right| = 0$.
  \end{center}   
  This is a quadratic equation $ax^{2}+bx+c=0$ with $deg(a) = deg(b) = 0$, and $deg(c) > 0$. Given this, the discriminant of the quadratic is nonzero, and there are two distinct roots $x_{1}$ and $x_{2}$, the first with degree zero, and the second with positive degree. We set $\tilde{a}_{-1,1} = \tilde{a}_{1,-1} = x_{1}$. Note the above $3 \times 3$ principal submatrix being symmetric implies
  \begin{center}
    $\left|\begin{array}{cc} \tilde{a}_{-1,0} & \tilde{a}_{-1,1} \\ \tilde{a}_{0,0} & \tilde{a}_{0,1} \end{array}\right|$ \hspace{.1 in} $=$ \hspace{.1 in } $\left|\begin{array}{cc} \tilde{a}_{0,-1} & \tilde{a}_{0,0} \\ \tilde{a}_{1,-1} & \tilde{a}_{1,0} \end{array}\right|$.
  \end{center}
  The degree of these determinants must be zero, because if it was not, it would be impossible for the determinant of the $3 \times 3$ principal submatrix to be zero.
  
  Every term $\tilde{a}_{i,1}$ with $i < -1$ and $\tilde{a}_{i,-1}$ with $i > 1$ is determined by the need for the matrices
  \begin{center}
    $\left(\begin{array}{ccc} \tilde{a}_{-1,-1} & \tilde{a}_{-1,0} & \tilde{a}_{-1,1} \\ \tilde{a}_{0,-1} & \tilde{a}_{0,0} & \tilde{a}_{0,1} \\ \tilde{a}_{i,-1} & \tilde{a}_{i,0} & \tilde{a}_{i,1} \end{array}\right)$ \hspace{.1 in} and \hspace{.1 in} $\left(\begin{array}{ccc} \tilde{a}_{i,-1} & \tilde{a}_{i,0} & \tilde{a}_{i,1} \\ \tilde{a}_{0,-1} & \tilde{a}_{0,0} & \tilde{a}_{0,1} \\ \tilde{a}_{1,-1} & \tilde{a}_{1,0} & \tilde{a}_{1,1} \end{array}\right)$  
  \end{center} 
  to be, respectively, singular. That every such term has degree zero follows from the $2 \times 2$ minors discussed above having degree zero.
  
  Every column in $\tilde{B}_{2}$ can be written as a linear combination of the left two: 
  \begin{center}
    $\lambda_{j} \tilde{\textbf{b}}_{0} + \mu_{j} \tilde{\textbf{b}}_{1} = \tilde{\textbf{b}}_{j}$.
  \end{center}
  We use these relations to define the entries $\tilde{a}_{i,j}$ with $i < 0$ and $j > 1$:  
  \begin{center}
    $\lambda_{j} \tilde{a}_{i,0} + \mu_{j} \tilde{a}_{i,1} = \tilde{a}_{i,j}$.
  \end{center}  
  We similarly use the right two columns of $\tilde{B}_{1}$ to define the terms $\tilde{a}_{i,j}$ with $i > 0, j < -1$. This determines a rank two matrix $\tilde{A}$. We must verify the matrix is symmetric, and is a lift of $A$.
  
  Suppose $i < 0$. We must verify that all terms $\tilde{a}_{i,j}$ with $j > 1$ have degree zero. We can write column $j$ as a linear combination of columns $-1$ and $1$:
  \begin{center}
    $\sigma_{j}\tilde{\textbf{a}}_{-1} + \rho_{j}\tilde{\textbf{a}}_{1} = \tilde{\textbf{a}}_{j}$.
  \end{center}
  As all the terms in row $0$ have degree zero, it cannot be that $\sigma_{j}$ and $\rho_{j}$ both have positive degree, and if their degrees were negative they must be equal. If the degrees were negative this would imply elements in $\tilde{B}_{2}$ with negative degree, which cannot be. If $deg(\rho_{j}) > 0$ while $deg(\sigma_{j}) = 0$, then $\tilde{B}_{2}$ would have a column outside the first where all elements have degree zero, which cannot be. So, we must have $0 = deg(\rho_{j}) \leq deg(\sigma_{j})$. As $\tilde{a}_{i,-1}$ has positive degree and $\tilde{a}_{i,1}$ has degree zero it must be that $\tilde{a}_{i,j}$ has degree zero as well. Identical reasoning gives us that all terms $\tilde{a}_{i,j}$ with $j < -1$ and $i > 0$ also have degree zero.    
  
  It remains to be proven that $\tilde{A}$ is symmetric. As $\tilde{B}_{1}$ and $\tilde{B}_{2}$ are symmetric, we must only prove $\tilde{a}_{i,j} = \tilde{a}_{j,i}$ when $ij < 0$. Suppose $j > 1$, and examine the two matrices 
  \begin{center}
    $\left(\begin{array}{ccc} \tilde{a}_{-1,-1} & \tilde{a}_{-1,0} & \tilde{a}_{-1,j} \\ \tilde{a}_{0,-1} & \tilde{a}_{0,0} & \tilde{a}_{0,j} \\ \tilde{a}_{1,-1} & \tilde{a}_{1,0} & \tilde{a}_{1,j} \end{array}\right)$, \hspace{.1 in} and \hspace{.1 in} $\left(\begin{array}{ccc} \tilde{a}_{-1,-1} & \tilde{a}_{-1,0} & \tilde{a}_{-1,1} \\ \tilde{a}_{0,-1} & \tilde{a}_{0,0} & \tilde{a}_{0,1} \\ \tilde{a}_{j,-1} & \tilde{a}_{j,0} & \tilde{a}_{j,1} \end{array}\right)$.
  \end{center}
  By construction 
  \begin{center}
    $\tilde{a}_{-1,0} = \tilde{a}_{0,-1}$, \hspace{.1 in} $\tilde{a}_{0,1} = \tilde{a}_{1,0}$, 
    
    $\tilde{a}_{0,j} = \tilde{a}_{j,0}$, \hspace{.1 in} and \hspace{.1 in} $\tilde{a}_{1,j} = \tilde{a}_{j,1}$.
  \end{center}
  As the above matrices are also singular we must have $\tilde{a}_{-1,j} = \tilde{a}_{j,-1}$. The proof that $\tilde{a}_{1,j} = \tilde{a}_{j,1}$ for $j < -1$ is essentially identical. From here verifying symmetry is a calculation:
  \begin{center}
    
    $\tilde{a}_{i,j} = \sigma_{j} \tilde{a}_{i,-1} + \rho_{j} \tilde{a}_{i,1} = \sigma_{j} \tilde{a}_{-1,i} + \rho_{j} \tilde{a}_{1,i}$
    
    $= \sigma_{j} (\sigma_{i} \tilde{a}_{-1,-1} + \rho_{i} \tilde{a}_{-1,1}) + \mu_{j} (\sigma_{i} \tilde{a}_{1,-1} + \rho_{i} \tilde{a}_{1,1})$
    
    $= \sigma_{i} (\sigma_{j} \tilde{a}_{-1,-1} + \rho_{j} \tilde{a}_{1,-1}) + \rho_{i} (\sigma_{j} \tilde{a}_{-1,1} + \rho_{j} \tilde{a}_{1,1})$
    
    $= \sigma_{i} (\sigma_{j} \tilde{a}_{-1,-1} + \rho_{j} \tilde{a}_{-1,1}) + \rho_{i} (\sigma_{j} \tilde{a}_{1,-1} + \rho_{j} \tilde{a}_{1,1})$
    
    $= \sigma_{i} \tilde{a}_{-1,j} + \rho_{i} \tilde{a}_{1,j} = \sigma_{i} \tilde{a}_{j,-1} + \rho_{i} \tilde{a}_{j,1} = \tilde{a}_{j,i}$.
  \end{center}
  
  So, $\tilde{A}$ is a rank two symmetric lift of $A$, and therefore $A$ has symmetric Kapranov rank two.
\end{proof}

\begin{lem}
  Suppose $A$ has the form  
  \begin{center}
    $\left(\begin{array}{ccccc} B_{1} & \textbf{0} & \textbf{0} & \textbf{0} & \textbf{0} \\ \textbf{0} & B_{2} & \textbf{0} & \textbf{0} & \textbf{0} \\ \textbf{0} & \textbf{0} & 0 & \textbf{0} & \textbf{0} \\ \textbf{0} & \textbf{0} & \textbf{0} & \textbf{0} & C \\ \textbf{0} & \textbf{0} & \textbf{0} & C^{T} & \textbf{0} \end{array}\right)$,
  \end{center}
  where $B_{1}, B_{2}$ are symmetric and positive, $C$ is nonnegative and does not contain a zero column, and either $C$ or both $B_{1}$ and $B_{2}$ have positive size. If $A$ has symmetric tropical rank two it has symmetric Kapranov rank two.
\end{lem}

\begin{proof}
  If $B_{1}$ and $B_{2}$ both have size zero, this is Lemma 1. If $C$ has size zero, this is Lemma 2. So, suppose $C$ has positive size, and at least one of $B_{1}$ and $B_{2}$ have positive size. The method of proof here is similar to the method used in the previous two lemmas. 
  
  By induction we may find a rank two symmetric lift for the upper-left matrix      
  \begin{center}  
    $\left(\begin{array}{ccc} B_{1} & \textbf{0} & \textbf{0} \\ \textbf{0} & B_{2} & \textbf{0} \\ \textbf{0} & \textbf{0} & 0 \end{array}\right)$,
  \end{center}
  and the lower-right matrix
  \begin{center}
      $\left(\begin{array}{ccc} 0 & \textbf{0} & \textbf{0} \\ \textbf{0} & \textbf{0} & C \\ \textbf{0} & C^{T} & \textbf{0} \end{array}\right)$.
  \end{center}
  Call these lifts $\tilde{B}$ and $\tilde{C}$, respectively. After possibly multiplying the left column and top row of $\tilde{C}$ by the same degree zero constant, we may assume the bottom-right entry of $\tilde{B}$ coincides with the top-left entry of $\tilde{C}$.
  
  The lifts $\tilde{B}$ and $\tilde{C}$ will be, respectively, the upper-left and lower-right parts of the lift $\tilde{A}$ we wish to construct. We number the rows and columns of $\tilde{A}$ in a manner similar to Lemmas 1 and 2, with the $a_{0,0}$ entry being the degree zero entry that must match in lifts $\tilde{B}$ and $\tilde{C}$. We must complete the lift $\tilde{A}$ by finding entries for the terms $a_{i,j}$ with $ij < 0$.
  
  We again start with the $3 \times 3$ principal submatrix:  
  \begin{center}
    $\left(\begin{array}{ccc} \tilde{a}_{-1,-1} & \tilde{a}_{-1,0} & \tilde{a}_{-1,1} \\ \tilde{a}_{0,-1} & \tilde{a}_{0,0} & \tilde{a}_{0,1} \\ \tilde{a}_{1,-1} & \tilde{a}_{1,0} & \tilde{a}_{1,1} \end{array}\right)$.
  \end{center}  
  We pick $\tilde{a}_{-1,1}$ and $\tilde{a}_{1,-1}$ such that this matrix is singular and $\tilde{a}_{-1,1} = \tilde{a}_{1,-1}$. As in Lemma 2, this means finding the roots of a quadratic $ax^{2}+bx+c$, but in this case all three coefficients have degree zero, which means the roots must also.
  
  Every term $\tilde{a}_{i,1}$ for $i < -1$, and $a_{i,-1}$ for $i > 1$, is then determined by the need for the matrices 
  \begin{center}
    $\left(\begin{array}{ccc} \tilde{a}_{i,-1} & \tilde{a}_{i,0} & \tilde{a}_{i,1} \\ \tilde{a}_{0,-1} & \tilde{a}_{0,0} & \tilde{a}_{0,1} \\ \tilde{a}_{1,-1} & \tilde{a}_{1,0} & \tilde{a}_{1,1} \end{array}\right)$ \hspace{.1 in} and \hspace{.1 in} $\left(\begin{array}{ccc} \tilde{a}_{-1,-1} & \tilde{a}_{-1,0} & \tilde{a}_{-1,1} \\ \tilde{a}_{0,-1} & \tilde{a}_{0,0} & \tilde{a}_{0,1} \\ \tilde{a}_{i,-1} & \tilde{a}_{i,0} & \tilde{a}_{i,1}\end{array}\right)$
  \end{center}
  to be, respectively, singular.
  
  Every column of $\tilde{C}$ can be written as a linear combination of the left two:
  \begin{center}
    $\lambda_{j}\tilde{\textbf{c}}_{0} + \mu_{j}\tilde{\textbf{c}}_{1} = \tilde{\textbf{c}}_{j}$.  
  \end{center}
  We use these relations to define the entries $\tilde{a}_{i,j}$ with $i < 0$ and $j > 1$:
  \begin{center}
    $\lambda_{j}\tilde{a}_{i,0} + \mu_{j}\tilde{a}_{i,1} = \tilde{a}_{i,j}$.
  \end{center}
  We similarly use the right two columns of $\tilde{B}$ to define the terms $\tilde{a}_{i,j}$ with $i > 0, j < -1$. This determines a rank two matrix $\tilde{A}$. We must verify the matrix is symmetric, and is a lift of $A$.
  
  We first prove $\tilde{A}$ is symmetric. By construction all terms of the form $\tilde{a}_{i,j}$ with $ij \geq 0$ satisfy $\tilde{a}_{i,j} = \tilde{a}_{j,i}$. Also, by construction $\tilde{a}_{1,-1} = \tilde{a}_{-1,1}$. Using these facts we note the matrices   
  \begin{center}
    $\left(\begin{array}{ccc} \tilde{a}_{i,-1} & \tilde{a}_{i,0} & x \\ \tilde{a}_{0,-1} & \tilde{a}_{0,0} & \tilde{a}_{0,1} \\ \tilde{a}_{1,-1} & \tilde{a}_{1,0} & \tilde{a}_{1,1} \end{array}\right)$ \hspace{.1 in} and \hspace{.1 in} $\left(\begin{array}{ccc} \tilde{a}_{-1,i} & \tilde{a}_{-1,0} & \tilde{a}_{-1,1} \\ \tilde{a}_{0,i} & \tilde{a}_{0,0} & \tilde{a}_{0,1} \\ x & \tilde{a}_{1,0} & \tilde{a}_{1,1} \end{array}\right)$
  \end{center} 
  are transposes. Therefore, $\tilde{a}_{i,1}$, the unique value of $x$ that makes the matrix on the left singular, is equal to $\tilde{a}_{1,i}$, the unique value of $x$ that makes the matrix on the right singular.
  
  Using these equalities we note the matrices
  \begin{center}    
    $\left(\begin{array}{ccc} \tilde{a}_{i,i} & \tilde{a}_{i,0} & x \\ \tilde{a}_{0,i} & \tilde{a}_{0,0} & \tilde{a}_{0,j} \\ \tilde{a}_{1,i} & \tilde{a}_{1,0} & \tilde{a}_{1,j} \end{array}\right)$ \hspace{.1 in} and \hspace{.1 in} $\left(\begin{array}{ccc} \tilde{a}_{i,i} & \tilde{a}_{i,0} & \tilde{a}_{i,1} \\ \tilde{a}_{0,i} & \tilde{a}_{0,0} & \tilde{a}_{0,1} \\ x & \tilde{a}_{j,0} & \tilde{a}_{j,1} \end{array}\right)$
  \end{center}  
  are also transposes. So, $\tilde{a}_{i,j}$, the unique value of $x$ that makes the matrix on the left singular, is equal to $\tilde{a}_{j,i}$, the unique value of $x$ that makes the matrix on the right singular. So, the matrix $\tilde{A}$ is symmetric.
  
  It remains to be proven that each $\tilde{a}_{i,j}$ with $ij < 0$ has degree zero. Suppose $i < 0, j > 0$. That $\tilde{a}_{i,j}$ has degree zero follows because the matrix
  \begin{center}
    $\left(\begin{array}{ccc} \tilde{a}_{i,i} & \tilde{a}_{i,0} & \tilde{a}_{i,j} \\ \tilde{a}_{0,i} & \tilde{a}_{0,0} & \tilde{a}_{0,j} \\ \tilde{a}_{j,i} & \tilde{a}_{j,0} & \tilde{a}_{j,j} \end{array}\right)$
  \end{center}   
  is singular, $\tilde{a}_{i,i}$ has positive degree, and all other terms that are not $\tilde{a}_{i,j} = \tilde{a}_{j,i}$ have degree zero. The only way this matrix could possibly be singular is if $\tilde{a}_{i,j}$ has degree zero. As our matrix is symmetric this completes the proof.
\end{proof}

\begin{lem}
  Let $A$ be a symmetric matrix with symmetric tropical rank two. After possibly a diagonal permutation $A$ has the block structure:
  \begin{center}
    $\left(\begin{array}{ccccc} \textbf{0} & \textbf{0} & \textbf{0} & \textbf{0} & \textbf{0} \\ \textbf{0} & B_{1} & \textbf{0} & \textbf{0} & \textbf{0} \\ \textbf{0} & \textbf{0} & B_{2} & \textbf{0} & \textbf{0} \\ \textbf{0} & \textbf{0} & \textbf{0} & \textbf{0} & C \\ \textbf{0} & \textbf{0} & \textbf{0} & C^{T} & \textbf{0} \end{array}\right)$,
  \end{center}
  where the matrices $B_{1}$ and $B_{2}$ are symmetric and positive, and the matrix $C$ is non-negative and has no zero columns. Each $\textbf{0}$ represents a zero matrix of the appropriate size.  It is possible that $A$ has no rows/columns with all $0$ entries, and so the first row/column blocks of $A$ may be empty. It is also possible that the matrices $B_{1}, B_{2}$ and $C$ may have size zero. The only exceptions being $A$ cannot be a matrix consisting of just one of the positive blocks ($B_{1}$ or $B_{2}$), nor can $A$ be the zero matrix. 
\end{lem}

\begin{proof}
  We begin by examining some properties of the matrix $A$ that are not dependent on it being symmetric. As defined in \cite{ds} the tropical convex hull of a set of real vectors $\{\textbf{v}_{1},\ldots,\textbf{v}_{m}\}$ is the set of all tropical linear combinations
  
  \begin{center}
    
    $c_{1} \odot \textbf{v}_{1} \oplus c_{2} \odot \textbf{v}_{2} \oplus \cdots \oplus c_{m} \odot \textbf{v}_{m}$ \hspace{.1 in} where $c_{1},\ldots,c_{m} \in \mathbb{R}$.
    
  \end{center}
  
  Theorem 4.2 from \cite{dss} states that the standard tropical rank of a real matrix is equal to one plus the dimension of the tropical convex hull of its columns. As the standard tropical rank of a matrix is equal to the standard tropical rank of its transpose, the standard tropical rank of a real matrix is also equal to one plus the dimension of the tropical convex hull of its rows. 
  
  We construct a matrix $A'$ from $A$ by adjoining the zero vector as the first column:
  
  \begin{center}
    
    $A' := \left(\begin{array}{cc} \textbf{0} & A \end{array}\right)$.
    
  \end{center}
  
  From $A'$ we construct the matrix $A^{+}$ by adjoining the zero row as the first row:
  
  \begin{center}
    
    $A^{+} := \left(\begin{array}{c} \textbf{0} \\ A' \end{array}\right) = \left(\begin{array}{cc} 0 & \textbf{0} \\ \textbf{0} & A \end{array}\right)$.
    
  \end{center}
  
  As the matrix $A$ has symmetric tropical rank two, by Corollary 1 it must also have standard tropical rank two. Every row of $A$ contains $0$ as its minimal entry, and so the tropical convex hull of the columns of $A'$ is equal to the tropical convex hull of the columns of $A$. Therefore, the standard tropical rank of $A'$ is two. As every column of $A'$ contains zero as its minimal entry the tropical convex hull of the rows of $A^{+}$ is equal to the tropical convex hull of the rows of $A'$. Therefore, the standard tropical rank of $A^{+}$ is two.

  We derive the asserted block decomposition of $A$ from the claim that any two columns of $A^{+}$ have either equal or disjoint cosupports, where the cosupport of a column is the set of positions where it does not have a zero. To prove this, observe that if this were not so $A^{+}$ would, possibly after permuting the rows and columns\footnote{These permutations are not required to be diagonal permutations.}, have the following submatrix, where $+$ denotes a positive entry, $?$ denotes a nonnegative entry, and the first column of the submatrix is taken from the first column of $A^{+}$. (Recall that each column of $A$ contains a zero entry.)

  \begin{center}

    $\left(\begin{array}{ccc} 0 & + & + \\ 0 & 0 & + \\ 0 & ? & 0 \end{array}\right)$

  \end{center}

  This $3 \times 3$ matrix is standard tropically nonsingular, which cannot be given $A^{+}$ has standard tropical rank two.

  We now return to properties dependent on $A$ being symmetric. If the diagonal entry $a_{i,i}$ of $A^{+}$ is positive, then, as $A^{+}$ is symmetric, for any entry $a_{j,i}$ with $j \neq i$ if $a_{j,i}$ is positive $a_{i,j}$ is as well, and this means columns $i$ and $j$ have equal cosupports. In particular, $a_{j,j}$ is positive. From this we see that the positive entries of column $i$, and the positive entries from columns with cosupports equal to column $i$, form a positive principal submatrix of $A^{+}$. After possibly a diagonal permutation, this submatrix is the submatrix $B_{1}$ of $A^{+}$. If $A^{+}$ contains additional positive diagonal entries outside of $B_{1}$ then, using identical reasoning, possibly after a diagonal permutation we have the submatrix $B_{2}$. There cannot be three positive diagonal blocks, for then we would be able to construct the $3 \times 3$ principal submatrix of $A$:
  \begin{center}
    $\left(\begin{array}{ccc} a & 0 & 0 \\ 0 & b & 0 \\ 0 & 0 & c \end{array}\right)$,
  \end{center}
  where $a,b,c > 0$. This matrix is not symmetrically tropically singular, and this would contradict that $A$ has symmetric tropical rank two. Note the difference here between the standard and the symmetric case. This $3 \times 3$ principal minor is not symmetrically tropically singular, but it is standard tropically singular. 

  After possibly another diagonal permutation we can arrange the columns and rows of $A^{+}$ so that, from left to right, the first columns are the zero columns, followed by the columns that contain $B_{1}$, followed by the columns that contain $B_{2}$. The remaining columns must all have a $0$ entry on the diagonal, and a positive entry $a_{i,j}$ for some $i \neq j$. Row $i$ obviously cannot be a zero row, nor can it intersect $B_{1}$ or $B_{2}$, and so must be below the submatrix $B_{2}$. Denote as $A''$ the submatrix formed by all columns to the right of $B_{2}$, and all rows below $B_{2}$.
  
  The submatrix $A''$ is symmetric, does not contain a zero row/column, and has $0$ along its diagonal. In particular its upper-left $1 \times 1$ principal submatrix is a zero matrix. Suppose the upper-left $k \times k$ principal submatrix of $A''$ is a zero matrix. If for some column $\textbf{a}'_{i}$ all the terms in $\textbf{a}'_{i}$ to the right of this $k \times k$ principal submatrix are $0$, then the diagonal permutation that switches indices $i$ and $k+1$ will construct an upper-left $(k+1) \times (k+1)$ principal submatrix that is a zero matrix. We continue this process until no such column $\textbf{a}'_{i}$ exists, in which case, given our result about either equal or disjoint cosupports, $A''$, and therefore $A$, has our desired block decomposition. 
  
  We note finally that $A$ cannot be just a positive block, because that would violate the assumption that the minimum value in every row/column is $0$. $A$ also cannot be the zero matrix, for then it would have symmetric tropical rank one.
\end{proof}

\begin{lem}
  If $A$ is a symmetric matrix normalized so the rows/columns have $0$ as their minimal entry, and $A^{+}$ is the augmented matrix
  
  \begin{center}
    
    $A^{+} = \left(\begin{array}{cc} 0 & \textbf{0} \\ \textbf{0} & A \end{array}\right)$,
    
  \end{center}
  
  then:
  
  \begin{enumerate}
    
  \item If $A$ has symmetric tropical rank two, so does $A^{+}$.
    
  \item If $A$ has symmetric Kapranov rank two, so does $A^{+}$.
    
  \end{enumerate}
\end{lem}

\begin{proof}[Proof of part (1)]
  Suppose $A$ has symmetric tropical rank two. We may assume that, possibly after a diagonal permutation, the matrix $A$ has the block decomposition given in Lemma 4. In the proof of Lemma 4 we demonstrated that if $A$ has symmetric tropical rank two, then $A^{+}$ has standard tropical rank two. By Proposition 2 the only way a symmetric matrix can have standard tropical rank two but not symmmetric tropical rank two is if a principal $3 \times 3$ submatrix is standard tropically singular but not symmetrically tropically singular. By assumption, $A$ has symmetric tropical rank two, so the only way $A^{+}$ could not is if a principal $3 \times 3$ submatrix of $A^{+}$ involving the initial zero row/column were tropically singular but not symmetrically tropically singular. The possible $3 \times 3$ principal submatrices of this type have the forms (where an element not specified as being $0$ is positive):
  \begin{center}
    $\left(\begin{array}{ccc} 0 & 0 & 0 \\ 0 & 0 & 0 \\ 0 & 0 & 0 \end{array}\right)$, \hspace{.1 in} $\left(\begin{array}{ccc} 0 & 0 & 0 \\ 0 & 0 & 0 \\ 0 & 0 & a_{i,i} \end{array}\right)$, \hspace{.1 in} $\left(\begin{array}{ccc} 0 & 0 & 0 \\ 0 & a_{i,i} & 0 \\ 0 & 0 & a_{j,j} \end{array}\right)$,
    
    $\left(\begin{array}{ccc} 0 & 0 & 0 \\ 0 & 0 & a_{i,j} \\ 0 & a_{j,i} & 0 \end{array}\right)$, \hspace{.1 in} $\left(\begin{array}{ccc} 0 & 0 & 0 \\ 0 & a_{i,i} & a_{i,j} \\ 0 & a_{j,i} & a_{j,j} \end{array}\right)$.
  \end{center}
  
  Of these possibilities the only one that is not necessarily symmetrically tropically singular is the last one. For this possibility, it could be standard tropically singular but not symetrically tropically singular if $a_{i,j} = a_{j,i} < a_{i,i},a_{j,j}$. If $A$ contains a zero row/column or the submatrix $C$ (from Lemma 4) has positive size then, possibly after a diagonal permutation, $A$ must contain the $3 \times 3$ matrix under examination as a principal submatrix, which could not be as $A$ has symmetric tropical rank two. If $A$ consists of two positive blocks and nothing else then, possibly after a diagonal permutation, $A$ has the following $3 \times 3$ principal submatrix:
  \begin{center}
    $\left(\begin{array}{ccc} a_{i,i} & a_{i,j} & 0 \\ a_{j,i} & a_{j,j} & 0 \\ 0 & 0 & a_{k,k} \end{array}\right)$,
  \end{center}
  where $a_{k,k} > 0$. If $a_{i,j} = a_{j,i} < a_{i,i}, a_{j,j}$ then this is a principal submatrix of $A$ that is not symmetrically tropically singular, which violates our assumption that $A$ has symmetric tropical rank two. This eliminates all possible ways $A^{+}$ could not have symmetric tropical rank two, and so it must.
\end{proof}

\begin{proof}[Proof of part (2)]
  If $A$ has symmetric Kapranov rank two then there exists a rank two symmetric lift which we will call $\tilde{A}$. From Lemma 4 we know $A$ must have two nonzero columns with disjoint cosupports. Denote as $\tilde{\textbf{a}}_{i}$ and $\tilde{\textbf{a}}_{j}$ the corresponding columns in $\tilde{A}$. If $\lambda, \mu \in \tilde{K}$ have degree zero but are otherwise generic, then the vector
  \begin{center}
    $\tilde{\textbf{v}} = \lambda\tilde{\textbf{a}}_{i} + \mu\tilde{\textbf{a}}_{j}$
  \end{center}
  has all degree zero terms. This is because as $\textbf{a}_{i}$ and $\textbf{a}_{j}$ have disjoint cosupports, the sum
  \begin{center}
    $\tilde{v}_{k} = \lambda \tilde{a}_{k,i} + \mu \tilde{a}_{k,j}$
  \end{center}
  involves at least one term, $a_{k,i}$ or $a_{k,j}$, of degree zero. If both have degree zero, then $\lambda$ and $\mu$ being generic guarantees we do not have cancellation of leading terms. So, $\tilde{v}_{k}$ has degree zero.

  The matrix formed by adjoining $\tilde{\textbf{v}}$ to $\tilde{A}$,
  \begin{center}
    $\tilde{A}' := \left(\begin{array}{cc} \tilde{\textbf{v}} & \tilde{A} \end{array}\right)$,
  \end{center}
  must have rank two. If we augment $\tilde{A}'$ by adding a row formed by the linear combination of rows $i$ and $j$ of $\tilde{A}'$ multiplied by $\lambda$ and $\mu$, respectively, then as $\tilde{A}$ is symmetric we get the symmetric matrix
  \begin{center}
    $\tilde{A}^{+} := \left(\begin{array}{cc} \tilde{a}_{0,0} & \tilde{\textbf{v}}^{T} \\ \tilde{\textbf{v}} & A \end{array}\right)$.
  \end{center}
  This matrix has rank two. The entry $\tilde{a}_{0,0}$ is:
  \begin{center}
    $\tilde{a}_{0,0} = \lambda \tilde{v}_{i} + \mu \tilde{v}_{k} = \lambda (\lambda \tilde{a}_{i,i} + \mu \tilde{a}_{i,j}) + \mu (\lambda \tilde{a}_{j,i} + \mu \tilde{a}_{j,j}) = \lambda^{2}\tilde{a}_{i,i} + 2\lambda\mu \tilde{a}_{i,j} + \mu^{2}\tilde{a}_{j,j}$.
  \end{center}
  For the final equality we use $\tilde{a}_{i,j} = \tilde{a}_{j,i}$. At least one of $a_{i,i}, a_{i,j}, a_{j,j}$ has degree zero. As $\lambda, \mu$ are generic we cannot have cancellation of leading terms, and therefore $\tilde{a}_{0,0}$ has degree zero.
  
  So, the above matrix is a rank two symmetric lift of $A^{+}$, and therefore $A^{+}$ has symmetric Kapranov rank two.
\end{proof}

We now have the lemmas we need to prove the major result of this section.

\begin{thm}
  A symmetric matrix $A$ has symmetric Kapranov rank two if and only if it has symmetric tropical rank two.
\end{thm}

\begin{proof}[Symmetric Kapranov rank two $\Rightarrow$ symmetric tropical rank two]
  If $A$ has symmetric Kapranov rank two then by Theorem 5 it cannot have symmetric tropical rank one. The symmetric tropical rank cannot be greater than the symmetric Kapranov rank, and so $A$ must have tropical rank two.
\end{proof}

\begin{proof}[Symmetric Kapranov rank two $\Leftarrow$ symmetric tropical rank two]
  Suppose $A$ is a symmetric matrix with symmetric tropical rank two. We may assume $A$ is in the form given by Lemma 4. If $A$ has only one zero row/column then by Lemma 3 $A$ has symmetric Kapranov rank two. If $A$ has no zero row/column then the matrix
  \begin{center}
    $A^{+} = \left(\begin{array}{cc} 0 & \textbf{0} \\ \textbf{0} & A \end{array}\right)$
  \end{center}
  has symmetric tropical rank two by Lemma 5, and therefore symmetric Kapranov rank two by Lemma 3. As $A^{+}$ has symmetric Kapranov rank two, by eliminating the first row/column from the lift we see $A$ has symmetric Kapranov rank two as well.
  
  If $A$ has more than one zero row/column we may proceed by induction on the number of such columns. In particular, $A$ must have the form
  \begin{center}
    $A = \left(\begin{array}{cc} 0 & \textbf{0} \\ \textbf{0} & A^{-} \end{array}\right)$,
  \end{center}
  where $A^{-}$ is a symmetric matrix with symmetric tropical rank two, with one fewer zero row/column than $A$, and therefore by induction $A^{-}$ has symmetric Kapranov rank two. By Lemma 5 $A$ has symmetric Kapranov rank two as well.
\end{proof}

Combining Theorems 4, 5, and 6 we see that the $r \times r$ minors of a symmetric $n \times n$ matrix form a tropical basis for $r = n$, $r = 2$, and $r = 3$.

\subsection{Rank three symmetric matrices}

As with standard tropical and Kapranov rank, the rank three case is a special boundary case. In \cite{z}, the author uses a technique called "the method of joints" to prove the $4 \times 4$ minors of a symmetric $5 \times 5$ matrix of indeterminates form a tropical basis. In the next section of this paper, we prove the $4 \times 4$ minors of an $n \times n$ matrix do \emph{not} form a tropical basis when $n > 12$.

In \cite{z} the author conjectures the $4 \times 4$ minors of an $n \times n$ symmetric matrix of indeterminates do form a tropical basis when $n \leq 12$, and believes the method of joints can be generalized to prove this.

\section{When the minors of a symmetric matrix do not form a tropical basis}

In this section we prove the $r \times r$ minors of an $n \times n$  symmetric matrix do \emph{not} form a tropical basis if $4 < r < n$. Nor do they form a tropical basis if $r = 4$ and $n > 12$.

\emph{Note}: All statements in this section about tropical ranks and symmetric tropical ranks for specific matrices can be verified using Maple code available at: 

http://www.math.utah.edu/\url{~}zwick/Dissertation/

\subsection{The foundational examples}

The examination of when the minors of a standard matrix do not form a tropical basis begins with a couple foundational examples. The same is true in the symmetric case. 

In \cite{dss}, Develin, Santos, and Sturmfels proved the cocircuit matrix of the Fano matroid,
\begin{center}
  $\left(\begin{array}{ccccccc} 1 & 1 & 0 & 1 & 0 & 0 & 0 \\ 0 & 1 & 1 & 0 & 1 & 0 & 0 \\ 0 & 0 & 1 & 1 & 0 & 1 & 0 \\ 0 & 0 & 0 & 1 & 1 & 0 & 1 \\ 1 & 0 & 0 & 0 & 1 & 1 & 0 \\ 0 & 1 & 0 & 0 & 0 & 1 & 1 \\ 1 & 0 & 1 & 0 & 0 & 0 & 1 \end{array}\right)$,
\end{center}
has tropical rank three but Kapranov rank four. If we permute the rows of this matrix with the permutation given by the disjoint cycle decomposition $(27)(36)(45)$ we get the symmetric matrix
\begin{center}  
  $\left(\begin{array}{ccccccc} 1 & 1 & 0 & 1 & 0 & 0 & 0 \\ 1 & 0 & 1 & 0 & 0 & 0 & 1 \\ 0 & 1 & 0 & 0 & 0 & 1 & 1 \\ 1 & 0 & 0 & 0 & 1 & 1 & 0 \\ 0 & 0 & 0 & 1 & 1 & 0 & 1 \\ 0 & 0 & 1 & 1 & 0 & 1 & 0 \\ 0 & 1 & 1 & 0 & 1 & 0 & 0 \end{array}\right)$.
\end{center}
While this symmetric matrix has standard tropical rank three, its symmetric tropical rank is four, and it's therefore \emph{not} an example of a matrix with symmetric tropical rank three but greater symmetric Kapranov rank.

This matrix can, however, be used to construct the following symmetric matrix with symmetric tropical rank three, but greater symmetric Kapranov rank:
\begin{center}
  $\left(\begin{array}{ccccccccccccc} 0 & 0 & 0 & 0 & 0 & 0 & 1 & 1 & 0 & 1 & 0 & 0 & 0 \\ 0 & 0 & 0 & 0 & 0 & 0 & 1 & 0 & 1 & 0 & 0 & 0 & 1 \\ 0 & 0 & 0 & 0 & 0 & 0 & 0 & 1 & 0 & 0 & 0 & 1 & 1 \\ 0 & 0 & 0 & 0 & 0 & 0 & 1 & 0 & 0 & 0 & 1 & 1 & 0 \\ 0 & 0 & 0 & 0 & 0 & 0 & 0 & 0 & 0 & 1 & 1 & 0 & 1 \\ 0 & 0 & 0 & 0 & 0 & 0 & 0 & 0 & 1 & 1 & 0 & 1 & 0 \\ 1 & 1 & 0 & 1 & 0 & 0 & 0 & 1 & 1 & 0 & 1 & 0 & 0 \\ 1 & 0 & 1 & 0 & 0 & 0 & 1 & 0 & 0 & 0 & 0 & 0 & 0 \\ 0 & 1 & 0 &0 & 0 & 1 & 1 & 0 & 0 & 0 & 0 & 0 & 0 \\ 1 & 0 & 0 & 0 & 1 & 1 & 0 & 0 & 0 & 0 & 0 & 0 & 0 \\ 0 & 0 & 0 & 1 & 1 & 0 & 1 & 0 & 0 & 0 & 0 & 0 & 0 \\ 0 & 0 & 1 & 1 & 0 & 1 & 0 & 0 & 0 & 0 & 0 & 0 & 0 \\ 0 & 1 & 1 & 0 & 1 & 0 & 0 & 0 & 0 & 0 & 0 & 0 & 0 \end{array}\right)$
\end{center}
The upper-right, and bottom-left, $7 \times 7$ submatrices of this $13 \times 13$ symmetric matrix are the symmetric version of the cocircuit matrix of the Fano matroid. This $13 \times 13$ matrix has symmetric tropical rank three. If it had symmetric Kapranov rank three then its upper-right $7 \times 7$ submatrix would have standard Kapranov rank three, and this is impossible.

In \cite{sh1} Shitov proved the matrix
\begin{center}
  $\left(\begin{array}{cccccc} 0 & 0 & 4 & 4 & 4 & 4 \\ 0 & 0 & 2 & 4 & 1 & 4 \\ 4 & 4 & 0 & 0 & 4 & 4 \\ 2 & 4 & 0 & 0 & 2 & 4 \\ 4 & 4 & 4 & 4 & 0 & 0 \\ 2 & 4 & 1 & 4 & 0 & 0 \end{array}\right)$,
\end{center}
has tropical rank four but Kapranov rank five. If we permute the rows of this matrix with the permutation $(135)(246)$, and the columns with the permutation $(16)(25)(34)$, we get the symmetric matrix
\begin{center}
  $\left(\begin{array}{cccccc} 0 & 0 & 2 & 4 & 1 & 4 \\ 0 & 0 & 4 & 4 & 4 & 4 \\ 2 & 4 & 2 & 4 & 0 & 0 \\ 4 & 4 & 4 & 4 & 0 & 0 \\ 1 & 4 & 0 & 0 & 2 & 4 \\ 4 & 4 & 0 & 0 & 4 & 4 \end{array}\right)$.
\end{center}
This symmetric $6 \times 6$ matrix has symmetric tropical rank four, and, as its Kapranov rank is five, its symmetric Kapranov rank is at least five. Applying Theorem 3 we see its symmetric Kapranov rank is exactly five. So, it is a $6 \times 6$ symmetric matrix with different symmetric tropical and symmetric Kapranov ranks.

\subsection{Generating larger examples}

We now demonstrate how, given a symmetric matrix with differing symmetric tropical and symmetric Kapranov rank, we can construct larger symmetric matrices with similar rank differences.

\begin{lem}
  Suppose $A$ is an $n \times n$ symmetric matrix with symmetric tropical rank $r$. Construct the $n \times (n+1)$ matrix $A'$ from $A$ by appending to the right of $A$ a column identical to the $n$th column of $A$: 
  \begin{center}
    
    $A' := \left(\begin{array}{cc} A & \textbf{a}_{n} \end{array}\right)$.
    
  \end{center}
  Construct the $(n+1) \times (n+1)$ matrix $A''$ from $A'$ by appending to the bottom of $A'$ a row identical to the $n$th row of $A'$. So, if $\textbf{a}_{n}'$ is the $n$th row of $A'$:
  \begin{center}
    
    $A'' := \left(\begin{array}{c} A' \\ \textbf{a}_{n}' \end{array}\right) = \left(\begin{array}{cc} A & \textbf{a}_{n} \\ \textbf{a}_{n}^{T} & a_{n,n} \end{array}\right)$.
    
  \end{center}
  The matrix $A''$ is symmetric, and has symmetric tropical rank $r$.
\end{lem}

\begin{proof}
  That the matrix $A''$ is symmetric given $A$ is symmetric is obvious from its construction. Also note columns $n$ and $n+1$ of $A''$ are identical, as are rows $n$ and $n+1$.

  Suppose $M$ is an $(r+1) \times (r+1)$ submatrix of $A''$ that inherits its row and column indices from $A''$. If the largest row and column indices of $M$ or both less than $n+1$, then $M$ is a submatrix of $A$, and by assumption $M$ is symmetrically tropically singular.

  If $M$ contains row indices $n$ and $n+1$, then $M$ has two identical rows, and by Proposition 4 is symetrically tropically singular. Same is true if $M$ contains column indices $n$ and $n+1$.
  
  If $M$ contains row index $n+1$ but not $n$, then it's identical to the $(r+1) \times (r+1)$ submatrix $M'$ replacing row index $n+1$ with $n$. Same applies to column indices, and so $M$ is equivalent to a submatrix of $A''$ with row and column indices both less than $n+1$, which again is symetrically tropically singular by assumption.
\end{proof}

\begin{cor}
  If the $r \times r$ minors of an $n \times n$ symmetric matrix of variables are not a tropical basis, then the $r \times r$ minors of an $(n+1) \times (n+1)$ symmetric matrix of variables are not a tropical basis.
\end{cor}

\begin{proof}
  That the $r \times r$ minors of an $n \times n$ symmetric matrix of variables are not a tropical basis is equivalent to the existence of an $n \times n$ symmetric matrix with symmetric tropical rank $r-1$, but greater symmetric Kapranov rank. If $A$ is such a matrix, then, by Lemma 6 above, there exists an $(n+1) \times (n+1)$ matrix $A''$ with symmetric tropical rank $r-1$ containing $A$ as a principal submatrix. If $A''$ had symmetric Kapranov rank $r-1$ so would $A$, and so the symmetric Kapranov rank of $A''$ must be greater than $r-1$. This implies the $r \times r$ minors of an $(n+1) \times (n+1)$ symmetric matrix of variables are not a tropical basis.
\end{proof}

\begin{lem}
  Suppose $A$ is an $n \times n$ symmetric matrix with tropical rank $r$. Construct the $(n+1) \times (n+1)$ matrix $A'$ from $A$ by choosing a number $M$ that is less than any entry of $A$, a number $P$ that is greater than any entry of $A$, and defining
  \begin{center}
    $A' = \left(\begin{array}{ccc|c} & & & P \\ & A & & \vdots \\ & & & P \\ \hline P & \cdots & P & M \end{array}\right)$.
  \end{center}
  The matrix $A'$ has symmetric tropical rank $r+1$.
\end{lem}

\begin{proof}
  As $A$ has symmetric tropical rank $r$ there is an $r \times r$ submatrix of $A$ that is not symmetrically tropically singular. Let $a_{1},\ldots,a_{r}$ denote the rows of $A$ that define this submatrix, $b_{1},\ldots,b_{r}$ denote the columns of $A$ that define this submatrix, and $D$ denote the submatrix's tropical determinant. The tropical determinant of the $(r+1) \times (r+1)$ submatrix of $A'$ defined by the rows $a_{1},\ldots,a_{r},a_{n+1}$, and the columns $b_{1},\ldots,b_{r},b_{n+1}$ must, given the definitions of $P$ and $M$, be equal to $D \odot M$, and the submatrix must be nonsingular. So, the tropical rank of $A'$ must be at least $r+1$.
  
  Take any $(r+2) \times (r+2)$ submatrix of $A'$. If it is a submatrix of $A$ then, as $A$ has tropical rank $r$, it must be singular. If the submatrix is formed from row $n+1$ of $A'$, but not column $n+1$, then we can see it must be tropically singular by taking a row expansion along the submatrix's bottom row, and noting that every $(r+1) \times (r+1)$ submatrix of $A$ is tropically singular. Similarly, if the submatrix is formed from column $n+1$ of $A'$, but not row $n+1$, the submatrix must be tropically singular. Finally, if the submatrix is formed from row $n+1$ and column $n+1$ then, given the definitions of $P$ and $M$, every tropical product of terms that equals the tropical determinant must involve the term $a_{n+1,n+1} = M$, and singularity of the $(r+2) \times (r+2)$ submatrix follows from the fact that every $(r+1) \times (r+1)$ submatrix of $A$ is tropically symetrically singular. So, the symmetric tropical rank of $A'$ is at most $r+1$, is therefore $r+1$.
\end{proof}

\begin{cor}
  If the $r \times r$ minors of an $n \times n$ symmetric matrix of variables are not a tropical basis, then the $(r+1) \times (r+1)$ minors of an $(n+1) \times (n+1)$ symmetric matrix of variables are not a tropical basis.
\end{cor}

\begin{proof}
  Suppose $A$ is an $n \times n$ symmetric matrix with symmetric tropical rank $r-1$ but greater symmetric Kapranov rank. By Lemma 7 the $(n+1) \times (n+1)$ matrix $A$ has symmetric tropical rank $r$. Suppose $\tilde{A}'$ is a rank $r$ symmetric lift of $A'$
  \begin{center}
    $\tilde{A}' = \left(\begin{array}{ccc|c} & & & \tilde{a}_{1,n+1} \\ & \tilde{A} & & \vdots \\ & & & \tilde{a}_{n,n+1} \\ \hline \tilde{a}_{n+1,1} & \cdots & \tilde{a}_{n+1,n} & \tilde{a}_{n+1,n+1} \end{array}\right)$.
  \end{center}
  If we multiply row $n+1$ of $\tilde{A}'$ by $\tilde{a}_{n,n+1}/\tilde{a}_{n+1,n+1}$ and subtract it from row $n$, and then multiply column $n+1$ by the same constant and subtract it from column $n$, we get the matrix
  \begin{center}
    $\tilde{B}' = \left(\begin{array}{ccc|c} & & & \tilde{a}_{1,n+1} \\ & \tilde{B} & & \vdots \\ & & & 0 \\ \hline \tilde{a}_{n+1,1} & \cdots & 0 & \tilde{a}_{n+1,n+1} \end{array}\right)$.
  \end{center}
  The matrix $\tilde{B}'$ is symmetric, has the same rank as $\tilde{A}'$, and from the definitions of $P$ and $M$ we see the degrees of the elements in $\tilde{B}'$ are the same as the corresponding elements in $\tilde{A}'$. Continuing in this manner, we can construct a symmetric matrix
  \begin{center}
    $\tilde{C}' = \left(\begin{array}{ccc|c} & & & 0 \\ & \tilde{C} & & \vdots \\ & & & 0 \\ \hline 0 & \cdots & 0 & \tilde{a}_{n+1,n+1} \end{array}\right)$.
  \end{center}
  where $\tilde{C}'$ has the same rank as $\tilde{A}'$, and the degrees of all elements in $\tilde{C}$ are the same as the corresponding elements in $\tilde{A}$. If this matrix has rank $r$, then the rank of $\tilde{C}$ must be $r-1$, but then $\tilde{C}$ would be a symmetric lift of $A$ with rank $r-1$, which cannot be. So, there is no rank $r$ symmetric lift of $A'$, and therefore $A'$ has symmetric troical rank $r$ but greater symmetric Kapranov rank, which means the $(r+1) \times (r+1)$ minors of an $(n+1) \times (n+1)$ symmetric matrix of indeterminates do not form a tropical basis.
\end{proof}

Combining the foundational examples from this section with the corollaries from this section, we obtain our theorem on when the $r \times r$ minors of an $n \times n$ symmetric matrix do not form a tropical basis.

\begin{thm}
  The $r \times r$ minors of an $n \times n$ symmetric matrix do not form a tropical basis when $4 < r < n$, or when $r = 4$ and $n > 12$.
\end{thm}

\begin{proof}
The follows immediately from inductively applying Corollaries 3 and 4 to the foundational examples from this section.
\end{proof}

Combining the results from Theorems 4, 5, 6, and 7, we get the overall result in Theorem 1.

\section{Further questions}

The most significant remaining question is whether the $r \times r$ minors of an $n \times n$ symmetric matrix of indeterminates form a tropical basis when $r = 4$ and $5 < n < 13$.

\newtheorem{question}{Question}
\begin{question}
  Do the $4 \times 4$ minors of an $n \times n$ symmetric matrix of indeterminates form a tropical basis when $5 < n < 13$?
\end{question}

As mentioned in section 3, in \cite{z} the author proves the $4 \times 4$ minors of a $5 \times 5$ symmetric matrix of indeterminates form a tropical basis. While a modified version of the approach used there could work for larger symmetric matrices, the number of cases that would need to be checked is significant, and checking them all would likely be arduous. In Section 3 of \cite{cjr} Chan, Jensen, and Rubei compute the set of $5 \times 5$ matrices of tropical rank at most 3, and of Kapranov rank at most 3, using the software Gfan. They then compare the sets and show they are equal. The most straightforward way to answer Question 1 might be to use Gfan or some other computational software package.

Along with Kapranov and tropical rank, in \cite{dss} Develin, Santos, and Sturmfels define a third notion of matrix rank in tropical geometry called the \emph{Barvinok rank}, and prove
\begin{center}
  tropical rank $\leq$ Kapranov rank $\leq$ Barvinok rank
\end{center}
and these inequalities can be strict.

In \cite{cc} Cartwright and Chan define the \emph{symmetric Barvinok rank}, along with two other notions of tropical rank for symmetric matrices, the \emph{star tree rank}, and the \emph{tree rank}, which have no analogs for general matrices. They prove
\begin{center}
  tree rank $\leq$ star tree rank $\leq$ symmetric Barvinok rank
\end{center}
and the inequality can be strict.

In \cite{z}, I prove
\begin{center}
  symmetric Kapranov rank $\leq$ symmetric Barvinok rank
\end{center}
and the inequality can be strict.

We have in total five distinct notions of rank for symmetric matrices coming from tropical geometry, and it would be interesting to see how they all fit together.

\begin{question}
  What inequalities exist between symmetric tropical, symmetric Kapranov, tree, star tree, and symmetric Barvinok ranks?
\end{question}
In particular, it remains to be investigated how symmetric tropical and symmetric Kapranov ranks relate to tree and star tree ranks, if they do at all. 

Finally, in \cite{kr} Kim and Roush prove that as $min(m,n)$ grows, the maximum difference between the tropical rank and the Kapranov rank of an $m \times n$ matrix grows without bound. A similar question can be asked in the symmetric case.

\begin{question}
  As $n$ grows, does the maximum difference between the symmetric tropical rank and the symmetric Kapranov rank of an $n \times n$ symmetric matrix grow without bound?
\end{question}

This question was answered in the affirmative for the difference between star tree rank and symmetric Barvinok rank in \cite{cc}, and in the affirmative for the difference between symmetric Kapranov rank and symmetric Barvinok rank in \cite{z}. If other inequalities exist between distinct notions of rank for symmetric matrices coming from tropical geometry, the same question, mutatis mutandis, would apply.

\end{document}